\documentclass[11pt]{article}
 \usepackage{amsmath,amsthm,amsfonts,amssymb}
\usepackage{verbatim}
\usepackage{latexsym}
\newtheorem{theorem}{Theorem}[section]
\newtheorem{lemma}[theorem]{Lemma}

\newtheorem{definition}[theorem]{Definition}
\newtheorem{corollary}[theorem]{Corollary}

\newtheorem{remark}[theorem]{Remark}
\newenvironment{prf} {{\bf Proof.}}{\hfill $\Box$}

\begin{document}
\title{Coorbit spaces and wavelet coefficient decay over general dilation groups}
\author{ Hartmut F\"uhr\\
\footnotesize\texttt{{fuehr@matha.rwth-aachen.de}} } \maketitle
\begin{abstract}
We study continuous wavelet transforms associated to matrix dilation
groups giving rise to an irreducible square-integrable quasi-regular
representation on ${\rm L}^2(\mathbb{R}^d)$. It turns out that
these representations are integrable as well, with respect to a wide
variety of weights, thus allowing to consistently quantify wavelet
coefficient decay via coorbit space norms. We then show that these
spaces always admit an atomic decomposition in terms of bandlimited
Schwartz wavelets. We exhibit spaces of Schwartz functions contained
in all coorbit spaces, and dense in most of them. We also present an
example showing that for a consistent definition of coorbit spaces,
the irreducibility requirement cannot be easily dispensed with.

 We then address the question how to predict wavelet coefficient decay
from vanishing moment assumptions. To this end, we introduce a new
condition on the open dual orbit associated to a dilation group: If the orbit is {\em temperately embedded}, it
is possible to derive rather general weighted mixed ${\rm
L}^{p}$-estimates for the wavelet coefficients from vanishing
moment conditions on the wavelet and the analyzed function. These
estimates have various applications: They provide very
explicit admissibility conditions for wavelets and integrable vectors, as well as sufficient
criteria for membership in coorbit spaces. As a further consequence, one obtains a transparent way of
identifying elements of coorbit spaces with certain (cosets of) tempered distributions.

 We then show that, for every dilation group in dimension two, the associated dual orbit is temperately embedded. In particular,
the general results derived in this paper apply to the shearlet group and its associated
family of coorbit spaces, where they complement and generalize the known results.
\end{abstract}



\section{Introduction}\label{introduction}

%
%

One of the starting points of wavelet theory was the introduction of
the continuous wavelet transform by Grossmann, Morlet and Paul
\cite{GrMoPa}. Building upon the group-theoretic understanding of
this transform, higher-dimensional analogues of the continuous
wavelet transform were later introduced, starting with \cite{Mu},
followed by (amongst others) \cite{BeTa,Fu96}. Since then,
higher-dimensional wavelet systems (some with, some without
underlying group structure) have been studied extensively. However,
while the representation-theoretic aspects of these transforms are
by now well-understood \cite{BeTa,Fu96,LaWeWeWi,Fu_LN,Fu10,DeMDeV}, the
same cannot be said of their approximation-theoretic properties,
with the notable exception of some specific examples such as the
similitude groups or the shearlet group and its higher-dimensional relatives. It is the purpose of this paper to initiate a systematic
study of the approximation-theoretic aspects of higher-dimensional
continuous wavelet transforms, with particular emphasis on the role
of the underlying group.

To this end, the following question is addressed: Given a function
$f$ and an admissible dilation group $H$, what do we need to know
about $f$ and $H$ to predict a certain decay of the wavelet
coefficients? With the notable exception of the similitude group
(where these questions are closely related to the theory of
isotropic homogeneous Besov spaces) and the shearlet group
(including some higher-dimensional relatives), this question has not
been studied systematically in higher dimensions.

Let us quickly introduce the basic notions of continuous wavelet
transforms in dimension $d>1$. For more details and background, we
refer the reader to the cited papers, and to Section
\ref{sect:disc_series}. We fix a closed matrix group $H < {\rm
GL}(d,\mathbb{R})$, the so-called {\bf dilation group}, and let $G =
\mathbb{R}^d \rtimes H$, which can be viewed as the group of affine
mappings generated by $H$ and all translations. Elements of $G$ are
denoted by pairs $(x,h) \in \mathbb{R}^d \times H$, and the product
of two group elements is given by $(x,h)(y,g) = (x+hy,hg)$. The 
explicit formula for the inverse of a group element $(x,h) \in G$ is given by  $(x,h)^{-1} = (-h^{-1}x,h^{-1})$. 
$G$ acts unitarily on ${\rm L}^2(\mathbb{R}^d)$ by the {\bf quasi-regular
representation} defined by
\begin{equation} \label{eqn:def_quasireg}
[\pi(x,h) f](y) = |{\rm det}(h)|^{-1/2} f\left(h^{-1}(y-x)\right)~.
\end{equation}
We assume that $H$ is an {\bf admissible dilation group}, which
in the context of this paper means that $\pi$ is an {\bf irreducible square-integrable
representation}. Square-integrability is defined by the requirement 
that there exists at least one nonzero $\psi
\in {\rm L}^2(\mathbb{R}^n)$ such that the matrix coefficient
\[
(x,h) \mapsto \langle \psi, \pi(x,h) \psi \rangle
\] is in ${\rm L}^2(G)$, which is the ${\rm L}^2$-space associated
to a left Haar measure on $G$. Such a $\psi$ is called {\bf admissible vector}; recall that
in the irreducible case it follows that the 
associated wavelet transform
\[
\mathcal{W}_\psi : {\rm L}^2(\mathbb{R}^d) \ni f \mapsto \left(
(x,h) \mapsto \langle f, \pi(x,h) \psi \rangle \right)
\] is a scalar multiple of an isometry, which gives rise
to the {\bf wavelet inversion formula}
\begin{equation} \label{eqn:wvlt_inv}
f = \frac{1}{c_\psi} \int_G \mathcal{W}_\psi f(x,h) \pi(x,h) \psi 
d\mu_G(x,h)~.
\end{equation}
This formula also explains the interest in admissible vectors, which
can be interpreted as basic building blocks with the property that
every function in ${\rm L}^2(\mathbb{R}^d)$ has a continuous
expansion in the wavelet system $(\pi(x,h) \psi)_{(x,h) \in G}$.

As already mentioned, the ${\rm L}^2$-theoretic aspects of this type
of decomposition are by now quite well understood, mostly thanks to
${\rm L}^2$-specific tools such as the Plancherel formula. Of
particular importance for an understanding of the wavelet transform
is the {\em dual action}, which is just
 the (right) linear action $\mathbb{R}^d \times H \ni(\xi,h) \mapsto
 h^T \xi$: By the results of \cite{Fu96,Fu10}, $H$ is admissible iff the dual action has a single open orbit
 $\mathcal{O} = \{ h^T \xi_0 : h \in H \} \subset \mathbb{R}^d$ of full measure (for some $\xi_0 \in \mathcal{O}$), such
 that in addition the stabilizer group $H_{\xi_0} = \{ h \in H : h^T \xi_0 = \xi_0 \}$ is
 compact. (This condition does of course not depend on the exact choice of $\xi_0 \in
 \mathcal{O}$.)
 The dual open orbit will play a key role in this paper as well.

Let me now sketch the contents of the paper. Section 2 puts the
question of wavelet coefficient decay on a proper
functional-analytic foundation, by showing that every admissible
dilation group gives rise to an {\em integrable} group
representation (see Theorem \ref{thm:ex_int_vec}, which is an
extension of a result in \cite{KaTa}). This observation allows to
invoke the powerful {\bf coorbit theory} developed by Feichtinger
and Gr\"ochenig \cite{FeiGr0,FeiGr1,FeiGr2,Gr}, which provides a
consistent and robust way of quantifying wavelet coefficient decay
by imposing norms on the coefficients. Here we focus on the spaces
$Co {\rm L}^{p,q}_v$ of (Besov-type) coorbit spaces associated to weighted
mixed ${\rm L}^{p}$-space. A particular asset of the coorbit approach is the
existence of discretizations and atomic decompositions, which are
available once certain additional technical conditions (involving
oscillation estimates) are met. In Lemma \ref{lem:osc}, we prove
that these conditions hold true if the analyzing wavelet is a
Schwartz function that is bandlimited to a compact subset of the
dual open orbit $\mathcal{O}$. Hence {\em any} such function can be
used to define an atomic decomposition (Theorem \ref{thm:at_dec}),
simultaneously valid in all coorbit spaces.

While these observations are interesting (and somewhat surprising in
their generality), they do not answer the question how to predict
wavelet coefficient decay in a fully satisfactory manner, at least
for two reasons: The first one is the obvious objection that the
coorbit definition only provides a circular answer to the question
why a function has good coefficient decay. In other words, coorbit
theory allows to {\em identify} functions that have good decay, in a
functional-analytically sound and consistent way, but it does not a
priori help {\em characterize} them.  A somewhat more subtle point
is that by definition, elements of coorbit space are contained in
the dual $\mathcal{H}_{1,w}^\sim$ of a certain Banach space of
functions on $G$. In order to understand which Fourier-analytic properties
of a signal result in a certain type of wavelet coefficient decay, it therefore seems desirable
to have an explicit embedding of coorbit spaces into, say, the space of tempered distributions,
rather than viewing them as elements of a fairly abstract dual space.

It is the purpose of the remaining sections of the paper to provide more concrete answers. To this end, we generalize a concept which may be considered as one of the cornerstones of wavelet analysis in one dimension, namely  the simple observation that vanishing moments of a wavelet, together with (local) smoothness of the analyzed signal, results in wavelet coefficient decay. 
For general dilation groups $H$, the correct generalization of this notion turns out to be that of {\em vanishing moments in $\mathcal{O}^c$,} (see Definition \ref{defn:van_mom}), where $\mathcal{O}^c$ denotes the complement of the dual orbit. Roughly speaking, the points of $\mathcal{O}^c$ can be understood as  the ``blind spots'' of the wavelet transform; i.e., they correspond to frequency information that the wavelet transform typically cannot resolve very well. Thus intuition suggests that functions that have little or no frequency content in or nearby those areas can be well approximated, and it is the purpose of Section \ref{sect:mod_embed} to make this intuition work. 

It turns out that under additional technical assumptions (formulated in the notion {\em temperately embedded dual orbit}, see Definition \ref{defn:mod_embedded}), it is indeed possible to predict 
 membership in a coorbit space from suitable smoothness and vanishing moment assumptions (Theorems \ref{thm:central}, \ref{cor:coorb_vm}). Moreover, the same arguments allow to easily identify those functions that qualify as analyzing wavelets for coorbit spaces. In particular, we show that there exist such vectors in $C_c^\infty(\mathbb{R}^d)$ (Theorem \ref{cor:L1_adm_cond_vm}). Finally, we introduce a natural closed subspace $\mathcal{F}^{-1} \mathcal{S}(\mathcal{O})$ of Schwartz functions that is continuously, and for $p,q < \infty$ densely, embedded in
$Co {\rm L}^{p,q}_v(G)$. A useful byproduct of the embedding is the dual embedding allowing to identify elements of coorbit spaces with (cosets of) tempered distributions.

In the final section of the paper, we verify that for all admissible dilation groups in dimension two, the associated open dual orbits are temperately embedded, thus making the results of the previous sections available for all wavelets in dimension two.  For the shearlet case, these results considerably extend the
previously derived results in \cite{DaKuStTe,DaStTe10,DaStTe11,DaStTe12}.

\section{Coorbit spaces associated to admissible dilation groups} \label{sect:disc_series}

In this section, we introduce the necessary notations, and comment on the relationship of
integrability versus square-integrability for quasi-regular
representations.
In the following, we simply write $\int \cdot dx$ for integration
against left Haar measure. The left Haar measure of $G$ is then
expressed as $d(x,h) = |\det(h)|^{-1} dx dh$, and the modular
function of $G$ is given by $\Delta_G(x,h) = \Delta_H(h)
|\det(h)|^{-1}$. Given $f \in {\rm L}^1(\mathbb{R}^d)$, its Fourier transform is defined as
\[
 \mathcal{F}(f)(\xi) := \widehat{f}(\xi) := \int_{\mathbb{R}^d} f(x) e^{-2\pi i \langle x,\xi \rangle} dx~,
\] with $\langle \cdot, \cdot \rangle$ denoting the euclidean scalar product on $\mathbb{R}^d$. We will use the same symbol $\mathcal{F}$ for the Fourier transform of tempered dsitributions.
 For any subspace $X \subset
\mathcal{S}'(\mathbb{R}^d)$, we let $\mathcal{F}^{-1} X$ denote its
inverse image under the Fourier transform.

 In order to avoid cluttered notation, we will occasionally use
the symbol $ X \preceq Y$ between expressions $X,Y$ involving one or
more functions or vectors in $\mathbb{R}^d$  if there exists a constant $C>0$,
independent of the functions and vectors occurring in $X$ and $Y$, such that
 $X \le CY$. We let $|\cdot|: \mathbb{R}^d \to
\mathbb{R}$ denote the euclidean norm on $\mathbb{R}^d$. This choice has
proven particularly convenient in particular for the estimates in Section \ref{sect:mod_embed},
although it is clear that any other choice of norm would just affect the constants. 
 We let $\| \cdot \|:
\mathbb{R}^{d \times d} \to \mathbb{R}^+$ denote an arbitrary norm.
If we wish to explicitly refer to the operator norm
of the induced linear operator $(\mathbb{R}^d, |\cdot|) \to  (\mathbb{R}^d, |\cdot |)$, we will use
the notation $\| g \|_\infty$. 

We will (somewhat inconsistently)
always use $|\alpha| = \sum_{i=1}^d \alpha_i$ for multiindices
$\alpha \in \mathbb{N}_0^d$, since in the following no serious confusion can arise
between vectors and multiindices.
For $r,m>0$, we let
\[
| f |_{r,m} = \sup_{x \in \mathbb{R}^d, |\alpha| \le r} (1+|x|)^{m}
|\partial^\alpha f (x)|~.
\] denote the associated Schwartz norm of a function $f: \mathbb{R}^d \to \mathbb{C}$ with suitably many partial derivatives.

Recall that a weight on a general locally compact group $G$ is by definition
a continuous function $v: G \to \mathbb{R}^+$ satisfying the
submultiplicativity condition $v(g_1g_2) \le v(g_1)v(g_2)$, for all $g_1,g_2 \in G$. For a weight $v$ on a semidirect product group $G = \mathbb{R}^d \times H$, one always has the estimate $v(x,h) \le v(x,0) v(0,h)$. In the following, we will consider weights that grow at most polynomially on the translation part, i.e. obeying 
\[
v(x,h) \le (1+|x|)^s w(h)
\] with a suitable weight $w: H \to \mathbb{R}^+$. 
One possible class of such weights are of the form $v(x,h) = (1+|x| + \|h\|_\infty)^s w(h)$, with a weight $w: H \to \mathbb{R}^+$.
Here submultiplicativity is checked by the computation
\begin{eqnarray*}
 v((x,h)(y,g)) & = & v(x+hy,hg)  \\
 & = & (1+|x+hy|+\|hg\|_{\infty})^s w(hg) \le (1+|x|+\|h\|_{\infty} + \|h\|_\infty \|g\|_\infty)^s w(h) w(g) \\
  & \le & (1+|x|+\|h\|_\infty)^s (1+|y|+\|g\|_\infty)^s w(h) w(g)~,
\end{eqnarray*}
and furthermore, we have the estimate
\begin{equation} \label{eqn:ineq_ws}
 v(x,h) \le (1+|x|)^s (1+\|h\|_\infty)^s w(h) = (1+|x|)^s w_s(h)~,
\end{equation} with the weight $w_s: H \to \mathbb{R}^+$, $w_s(h) = (1+\|h\|_\infty)^s w(h)$. 

We define, for $1 \le p,q < \infty$, the weighted mixed $L^p$-space by 
\[
L^{p,q}_v (G) = \left\{ F: G \to \mathbb{C}~:~\int_H \left(
\int_{\mathbb{R}^d} |F(x,h)|^p v(x,h)^p dx \right)^{q/p} \frac{dh}{|{\rm det}(h)|} < \infty
\right\}~,
\]
with the obvious norm, and the usual conventions regarding
identification of a.e. equal functions. We write ${\rm L}^p_v(G) =
{\rm L}^{p,p}_v(G)$. The corresponding spaces for $p=\infty$ and/or
$q = \infty$ are defined by replacing integrals with essential
suprema. We will also use
\[
{\rm L}^p_s(\mathbb{R}^d) = \left\{ f \mbox{
Borel-measurable}~:~\int_{\mathbb{R}^d} |f(x)|^p (1+|x|)^{sp} dx <
\infty \right\}~,
\] again with the obvious norm. 
Of particular interest is the space of $v$-integrable vectors, defined as
\[
\mathcal{H}_{1,v} = \{ f \in {\rm L}^2(\mathbb{R}^d)~:~
\mathcal{W}_f f \in {\rm L}^1_v(G) \}~.
\]

 The following result is known for the case of trivial dual fixed
 groups and weight $w(x,h) = \Delta_G^{1/2}(h)$, see
\cite{KaTa}. The general case needs only little adjustments. In the following,
 $C_c^\infty(\mathcal{O})$ denotes the space of $f \in
C^\infty_c(\mathbb{R}^d)$ such that ${\rm supp}(f) =
\overline{f^{-1}(\mathbb{C} \setminus \{ 0 \})} \subset
\mathcal{O}$.
\begin{theorem} \label{thm:ex_int_vec}
Assume that $\pi$ is square-integrable, and $v$ is a weight fulfilling the estimate
$v(x,h) \le (1+|x|)^s w(h)$, for some weight $w$ on $H$. 
Then $\pi$ is $v$-integrable; in fact, 
if $\psi \in \mathcal{F}^{-1} C_c^\infty(\mathcal{O})$,
then $\mathcal{W}_\psi \psi \in L^1_v(G)$.
\end{theorem}

\begin{proof}
Let $\psi$ be as in the theorem. We first rewrite the wavelet
transform as
\[
\mathcal{W}_\psi \psi(x,h) = |\det(h)|^{-1/2} \left( \psi \ast (\pi(0,h)\psi^*) \right) (x) = |{\rm det}(h)|^{1/2}\left( \widehat{\psi} \cdot D_h \overline{\widehat{\psi}} \right)^{\vee}(x)
\] where we used the notations $\psi^*(x) = \overline{\psi(-x)}$, and $D_h g (\xi)  = g(h^{T} \xi)$. Thus we
can use our assumption on the weight $v$ to estimate
\[
\| \mathcal{W}_\psi \psi \|_{{\rm L}^1_v} \le
 \int_{H}  \left\| \left( \widehat{\psi} \cdot D_h \widehat{\psi}\right)^\vee \right\|_{{\rm L}^1_s}
 w(h) |\det(h)|^{-1/2} dh ~.
\]
Since the
transposed action of $H$ on $\mathbb{R}^d$ is smooth, the
vector-valued mapping
\[
h \mapsto \widehat{\psi} \cdot D_h \widehat{\psi} \in C_c^\infty(\mathcal{O})
\] is continuous with respect to the Schwartz topology.  Since the Fourier transform is an
automorphism of $\mathcal{S}(\mathbb{R}^d)$, and
$\mathcal{S}(\mathbb{R}^d) \subset {\rm L}^1_s(\mathbb{R}^d)$ continuously, it follows that the mapping
\[ h \mapsto \| (\widehat{\psi} \cdot D_h \widehat{\psi})^\vee \|_{{\rm L}^1_s}
 w(h) |\det(h)|^{-1/2}
\] is continuous as well. It is also compactly supported: Fixing any
$\xi_0 \in \mathcal{O}$, let $p_{\xi_0} : H \to \mathcal{O}, h
\mapsto h^T \xi_0$. Then $p_{\xi_0}$ is continuous, and the induced canonical bijection
between quotient $H_{\xi_0} \setminus H$ and orbit $\mathcal{O}$ is a homeomorphism. Here $H_{\xi_0}$ denotes the stabilizer of $\xi_0$, which is compact by assumption. But this implies that $K = p_{\xi_0}^{-1}({\rm
supp}(\widehat{\psi}))$ is compact. Now $\widehat{\psi}
\cdot D_h \widehat{\psi} \not= 0$ implies the existence of $h_1 \in K$ such that $h_1 h
\in K$ as well, thus finally $h \in K^{-1}K$, which is compact.

Since the Haar measure of the compact set $K^{-1}K$ is finite, it
follows that $\| \mathcal{W}_\psi \psi \|_{{\rm L}^1_v}<\infty$.
\end{proof}

Now let $Y$ be a Banach function space on $G$. We assume that $Y$
fulfills the conditions of \cite[2.2]{Gr}, i.e. it is a Banach
function space continuously embedded in ${\rm L}^1_{loc}(G)$, and
fulfilling certain compatibility conditions with convolution. We
remark that the spaces ${\rm L}^{p,q}_v(G)$, which are our main
concern,  all fall in this category. The following definition spells out the compatibility condition:
\begin{definition}
A weight $v_0$ is called {\bf control weight} for $Y$ if it satisfies 
\[
 v_0(x,h) = \Delta_G(x,h)^{-1} v_0((x,h)^{-1})~, 
\]
as well as 
 \[
\max \left( \|L_{(x,h)^{\pm 1}} \|_{Y \to Y},\| R_{(x,h)} \|_{Y \to Y},\|
R_{(x,h)^{-1}} \|_{Y \to Y} \Delta_G(x,h)^{-1} \right) \le v_0(x,h)
\] where $L_{(x,h)},R_{(x,h)}: Y \to Y$ are the left and right translation operators .
\end{definition}

Note that as a byproduct of the definition, control weights are bounded from below, since for arbitrary
$(x,h) \in G$, submultiplicativity of the operator norm yields
\[ 1 = \| L_{(0,{\rm id}} \|_{Y \to Y} \le \| L_{(x,h)} \|_{Y \to Y} \| L_{(x,h)^{-1}} \|_{Y \to Y}
 \le v_0(x,h)^2~. 
\]

Now general coorbit spaces, with respect to suitable Banach function
spaces $Y$ on $G$, are defined as follows \cite{FeiGr0}: Let $v_0$ denote a control weight 
for $Y$ that is bounded from below. Fix any $0 \not= \psi_0 \in {\rm L}^2(\mathbb{R}^d)$
with $\mathcal{W}_{\psi_0} \psi_0 \in {\rm L}^1_{v_0}(G)$, and let 
\[
 \mathcal{H}_{1,v_0} = \{ f \in {\rm L}^2(\mathbb{R}^d) ~:~ \mathcal{W}_{\psi_0} f \in {\rm L}^1_v(G) \}~,
\] endowed with the norm $\| f \|_{\mathcal{H}_{1,v_0}} = \| \mathcal{W}_{\psi_0} f \|_{{\rm L}^1_v}$. 
Then $\mathcal{H}_{1,v_0}$ is a Banach space, which is invariant under $\pi$. Denote by
$\mathcal{H}_{1,v_0}^{\sim}$ the conjugate dual of
$\mathcal{H}_{1,v_0}$. Thus $\mathcal{H}_{1,v_0} \subset {\rm
L}^2(\mathbb{R}^d) \subset \mathcal{H}_{1,v_0}^{\sim}$, and the
sesquilinear map ${\rm L}^2(\mathbb{R}^d) \times \mathcal{H}_{1,v_0}
\ni (f,g) \mapsto \langle f,g \rangle$ can be uniquely extended to
$\mathcal{H}_{1,v_0}^{\sim} \times \mathcal{H}_{1,v_0}$. Hence, if we
fix a nonzero element $\psi \in \mathcal{H}_{1,v_0}$, we can define
the continuous wavelet transform of $f \in \mathcal{H}_{1,v_0}^{\sim}$
by the same formula as for elements of ${\rm L}^2(\mathbb{R}^d)$.

Now the coorbit space associated to
$Y$ is defined by picking a nonzero $\psi \in \mathcal{H}_{1,v_0}$ and then letting 
\[
{\rm Co}(Y) = \{ f \in \mathcal{H}_{1,v_0}^{\sim} : \mathcal{W}_\psi f
\in Y \}
\] with the norm $\| f \|_{{\rm Co}(Y)} = \| \mathcal{W}_\psi f
\|_Y$. A cornerstone of coorbit theory is the fact that the space $Co Y$ thus defined is independent of the analyzing vector, as well as of the precise choice of control weight; see \cite{FeiGr1} for details. 

The following lemma notes that for weighted $L^{p,q}$-spaces of the type  $v_0 (x,h) = (1+|x|+\|h \|_\infty)^s w_0(h)$ considered here, there exists an explicitly computable control weight $v_0$ of the same type.  Hence all results in this paper which state weighted integrability properties of certain matrix coefficients can be applied to decide membership in coorbit spaces $Co(Y)$, but also to decide whether a wavelet is a suitable analyzing wavelet for $Co(Y)$; we just have to switch the weight. For these reasons we will not explicitly distinguish in the subsequent results between $v$ and $v_0$.

\begin{lemma} \label{lem:weight_control}
 Let $v: G \to \mathbb{R}^+$ denote a weight with $v(x,h) \le (1+|x|)^s w(h)$. There exists a control weight $v_0$ for $Y = {\rm L}^{p,q}_v(G)$ satisfying the estimate
 \begin{equation} \label{eqn:cont_weight_sep}
  v_0(x,h) \le (1+|x|)^s w_0(h)~,
 \end{equation} with $w_0 : H \to \mathbb{R}^+$ defined by 
 \begin{eqnarray*} \nonumber
  w_0(h) & = &    (w(h)+w(h^{-1})) \max \left(\Delta_G(0,h)^{-1/q}, \Delta_G(0,h)^{1/q-1} \right)  \\  & & \times \left(|{\rm det}(h)|^{1/q-1/p} + |{\rm det}(h)|^{1/p-1/q} \right) (1+\|h\|_\infty+\|h^{-1}\|_\infty)^s~.
 \end{eqnarray*} 
 Here we use the convention that $1/\infty = 0$. 
\end{lemma}
\begin{prf}
We start out by determining upper bounds for the translation norms in terms of $v$ and suitable powers of the modular and determinant functions. We first consider right translation and the case $p,q< \infty$. Given $F \in {\rm L}^{p,q}_v(G)$, we use the invariance properties of the various measures and functions to compute
\begin{eqnarray*}
 \| R_{(x,h)} F\|_Y^q & = & \int_H \left( \int_{\mathbb{R}^d} |F((y,g)(x,h))|^p v(y,g)^p dy \right)^{q/p} \frac{dg}{|{\rm det}(g)|} \\
 & = & \int_H \left( \int_{\mathbb{R}^d}|F(y+gx,gh)|^p v(y,g)^p dy \right)^{q/p} \frac{dg}{|{\rm det}(g)|} \\
& = & \int_H \left( \int_{\mathbb{R}^d}|F(y,gh)|^p v(y-gx,g)^p dy \right)^{q/p} \frac{dg}{|{\rm det}(g)|} \\
& = & \int_H \left( \int_{\mathbb{R}^d} |F(y,g)|^p \underbrace{v(y-gh^{-1}x,gh^{-1})^p}_{=v((y,g)(x,h)^{-1})^p} dy \right)^{q/p} \Delta_H(h)^{-1} \frac{dg}{|{\rm det}(gh^{-1})|} \\
& \le & \int_H \left( \int_{\mathbb{R}^d} |F(y,g)|^p v(y,g)^p v((x,h)^{-1})^p dy \right)^{q/p} \Delta_H(h)^{-1} \frac{dg}{|{\rm det}(gh^{-1})|} \\
& \le &  \Delta_G(x,h)^{-1}  v((x,h)^{-1})^q \| F \|_Y^q\\ &  \le & \Delta_G(0,h)^{-1} v_1(x,h)^q \| F \|_Y^q~,
\end{eqnarray*}
hence
\begin{equation} \label{eqn:weight_left_1}
 \| R_{(x,h)} \|_{Y \to Y} \le \Delta_G(0,h)^{-1/q} v(x,h)~,
\end{equation}
and consequently 
\begin{equation}
 \| R_{(x,h)^{-1}} \|_{Y \to Y} \Delta_G(x,h)^{-1} \le \Delta_G(0,h)^{1/q-1} v((x,h)^{-1})~.
\end{equation}The computation for left translation is similar, with slightly different result:
\begin{eqnarray*}
 \| L_{(x,h)} F \|_Y^q & = &  \int_H \left( \int_{\mathbb{R}^d} |F((x,h)^{-1}(y,g))|^p v(y,g)^p dy \right)^{q/p} \frac{dg}{|{\rm det}(g)|} \\
 & = & \int_H \left( \int_{\mathbb{R}^d} |F(h^{-1}y-h^{-1}x,h^{-1} g)|^p v(y,g)^p dy \right)^{q/p} \frac{dg}{|{\rm det}(g)|} \\
 & = & \int_H \left( |{\rm det}(h)| \cdot \int_{\mathbb{R}^d} |F(y,h^{-1} g))|^p v(hy+x,g)^p dy \right)^{q/p} \frac{dg}{|{\rm det}(g)|} \\
 & = & \int_H  \left( |{\rm det}(h)| \cdot \int_{\mathbb{R}^d} |F(y, g))|^p \underbrace{v(hy+x,hg)^p}_{= v((x,h)(y,g))^p} dy \right)^{q/p} |{\rm det}(h)|^{q/p} \frac{dg}{|{\rm det}(hg)|} \\
 & \le & |{\rm det}(h)|^{q/p-1} v(x,h)^q \| F \|_Y~, 
\end{eqnarray*} where we again used submultiplicativity of $v$. Thus 
\begin{equation}
 \| L_{(x,h)} \|_{Y \to Y} \le |{\rm det}(h)|^{1/p-1/q} v(x,h)~,
\end{equation}
and  
\begin{equation} \label{eqn:weight_right_2}
 \| L_{(x,h)^{-1}} \|_{Y \to Y} \le |{\rm det}(h)|^{1/q-1/p} v((x,h)^{-1})~.
\end{equation}

Now we introduce the weight 
\[
 v_1(x,h) = \left(1+|x|+|h^{-1} x| + \|h^{-1}\|_\infty + \|h\|_\infty\right)^s
 \]
 This is a symmetric submultiplicative weight on $G$: Symmetry is clear by construction, and submultiplicativity is verified by the calculation 
\begin{eqnarray*}
v_1((x,h)(y,g)) &= &  \left(1+|x+hy| + |(hg)^{-1} (x+hy)| + \|hg\|_\infty+ \|(hg)^{-1}\|_\infty\right)^s \\
 & \le & \left(1+|x|+\|h\|_\infty |y| + |g^{-1} y| + \|g^{-1}\|_\infty |h^{-1}x| + \|h\|_\infty \|g\|_\infty +
 \|h^{-1}\|_\infty \|g^{-1}\|_\infty\right)^s  \\
 & \le & \left(1+|x|+|h^{-1} x| + \|h^{-1}\|_\infty + \|h\|_\infty\right)^s \left(1+|y|+|g^{-1} y| + \|g^{-1}\|_\infty + \|g\|_\infty\right)^s ~.
\end{eqnarray*}

Furthermore, $v(x,h) \le v_1(x,h) w(h)$. Thus, if we let
\[
 v_2(x,h) = v_1(x,h) \max \left(\Delta_G(0,h)^{-1/q}, \Delta_G(0,h)^{1/q-1} \right) (w(h)+w(h^{-1})) \left(|{\rm det}(h)|^{1/p-1/q}  + |{\rm det}(h)|^{1/q-1/p} \right)~,
\]
the inequalities (\ref{eqn:weight_left_1}) through (\ref{eqn:weight_right_2}) yield 
\[
 \max \left( \|L_{(x,h)^{\pm 1}} \|_{Y \to Y},\| R_{(x,h)} \|_{Y \to Y},\|
R_{(x,h)^{-1}} \|_{Y \to Y} \Delta_G(x,h)^{-1} \right)  \le v_2(x,h)~.
\]
In addition, symmetry of $v_1$ and of the mapping 
\[ h \mapsto  (w(h)+w(h^{-1})) (|{\rm det}(h)|^{1/p-1/q}  + |{\rm det}(h)|^{1/q-1/p} )  \]
allow to verify that 
\[
 v_2(x,h) =  \Delta_G(x,h)^{-1} v_2((x,h)^{-1})~. 
\] Thus $v_2$ is indeed a control weight. 
Finally, the inequality
\[ 
  \left(1+|x|+|h^{-1} x| + \|h^{-1}\|_\infty + \|h\|_\infty\right)^s  \le (1+|x|)^s (1+ \|h\|_\infty+\| h^{-1} \|_\infty)^s
 \] yields the estimate (\ref{eqn:cont_weight_sep}). 

In the case $q= \infty, p < \infty$, we use the fact that right translation on $G$ maps null-sets to null-sets to obtain
\begin{eqnarray*}
 \| R_{(x,h)} F\|_Y & = & {\rm ess~sup}_{g \in H} \left( \int_{\mathbb{R}^d} |F((y,g)(x,h))|^p v(y,g)^p dy \right)^{1/p} \\
 & = & {\rm ess~sup}_{g \in H}  \left( \int_{\mathbb{R}^d} |F((y,g))|^p v((y,g)(x,h)^{-1})^p dy\right)^{1/p}  \\
 & \le & v((x,h)^{-1}) \| F \|_Y~.
\end{eqnarray*}
The remaining estimates are obtained in a similar fashion.
\end{prf}

\begin{remark} \label{rem:controlling_dependence}
 Since one is often interested in using a single wavelet to characterize whole classes of coorbit spaces, it is worthwhile to study the dependence of the control weight on the space $Y$. Fixing $s_0>0$ and a weight $w$ on $H$, the Lemma shows that one can choose a weight $m$ working simultaneously for all $1 \le p,q \le \infty$ and all $0 \le s \le s_0$. For instance, letting
 \begin{eqnarray*}
  m(x,h)  & = & (1+|x|)^{s_0}(w(h)+w(h)^{-1})  \max(1,\Delta_G(0,h)^{-1})  \\ & & \times (1+\| h \|_{\infty} + \|h^{-1}\|_\infty)^{s_0}  \left(|{\rm det}(h)| + |{\rm det}(h)|^{-1} \right)~.
 \end{eqnarray*} yields $v_1(x,h) \le m(x,h)$ for all control weights $v_1$ associated to $v$ via Lemma \ref{lem:weight_control}, as long as $w$ is fixed and $s \le s_0$. Here $ 1 \le p,q \le \infty$ are arbitrary. 

Thus a wavelet in ${\rm L}^1_m(G) \subset \mathcal{H}_{1,v_1}$ can be employed for the characterization of all coorbit spaces $Co (Y)$ with $Y = {\rm L}^{p,q}_v(G)$, as long as $s \le s_0$ holds. 
\end{remark}

It is fairly intuitive that the isotropic homogeneous Besov spaces
can be understood as coorbit spaces associated to the action of the
similitude group $H = \mathbb{R}^+ \times SO(d)$ as dilation group:
As soon as one chooses an isotropic window (such as the
$\phi$-functions of Frazier and Jawerth \cite{FrJa}), the continuous
wavelet transform is constant on the $SO(d)$-cosets. Hence the
rotation group can be omitted altogether. But then we are left with
the continuous wavelet transform with respect to scalar dilations,
and the $\phi$-transform characterization of \cite{FrJa} allows to
understand these spaces as coorbit spaces, either with respect to
the irreducible representation arising from the similitude group, or
with respect to the reducible representation arising from the scalar
action. In a similar way, the anisotropic Besov space studied in
\cite{Bo} can be read as coorbit spaces associated to cyclic or
one-parameter dilation groups. Note however that in these cases the
quasi-regular representation is reducible.

An admissible dilation group for which the associated
coorbit spaces have been studied extensively is the shearlet
group, see \cite{DaKuStTe,DaStTe10,DaStTe11,DaStTe12} and the references therein. We point out
however that typically, attention there is restricted to the spaces ${\rm
Co}(L^{p,p}_v)$, and only for weights $v$ of the type $v(x,h) = w(h)$.

We will next establish atomic decompositions for the coorbit spaces. For this purpose, we need the notion of oscillation:
\begin{definition}
 Let $U \subset G$ denote a relatively compact neighborhood of the identity, and $F: G \to \mathbb{C}$ any function. We let
\[
 {\rm osc}_U(F)(x) = \sup \{ |F(x)-F(xy)|: y \in U \}~.
\]
\end{definition}

The following lemma will help to effectively deal with the determinant factor in estimating the oscillation of the wavelet transform:
\begin{lemma} \label{lem:osc_hom} Let $G$ be an arbitrary locally compact group, and $v$ a weight on $G$. 
 Let $F \in {\rm L}^1_v(G)$, and $\varphi: G \to \mathbb{R}^+$ a continuous homomorphism. Let $U$ be a relatively compact neighborhood of the identity in $G$.
If both $F$ and ${\rm osc}_U F$ are in ${\rm L}^1_v(G)$, then $ {\rm osc}_U(F \cdot \varphi)
\in {\rm L}^1_{v/\varphi}(G)$.
\end{lemma}
\begin{proof}
We first note that
\begin{equation} \label{eqn:osc_hom}
 {\rm osc}_U(F \cdot \varphi)(x)
 \le \varphi(x) \left ( \sup_{y \in U} \varphi(y) \right) {\rm osc}_U
F(x) + \varphi(x) \left( \sup_{y \in U} |1-\varphi(y)| \right)
|F(x)|~,
\end{equation}
which follows from
\[
 |F(x)\varphi(x) - F(xy)\varphi(xy)| \le \varphi(xy) |F(x)-F(xy)| + \varphi(x) |1-\varphi(y)|  |F(x)|
\]
Our assumptions on $F$ now yield that the right-hand side of
(\ref{eqn:osc_hom}) is in ${\rm L}^1_{v/\varphi}(G)$
\end{proof}

The next lemma paves the way towards existence of atomic decompositions.
\begin{lemma} \label{lem:osc} Assume that $\pi$ is square-integrable, and $v$ a separable weight on $G$.
 If $\psi \in \mathcal{F}^{-1} C_c(\mathcal{O})$, there exists a neighborhood $U \subset G$ of the identity such that ${\rm osc}_U (\mathcal{W}_\psi \psi)
\in L^1_v(G)$.
\end{lemma}
\begin{proof}
Let $B_r(x) \subset \mathbb{R}^d$ denote the open ball with radius $r$ and center $x$. We write $\widetilde{\mathcal{W}}_\psi \psi(x,h) =
|{\rm det}(h)|^{1/2} \mathcal{W}_\psi \psi (x,h)$, and $v'(x,h) =
|{\rm det}(h)|^{-1/2} v(x,h)$. By Lemma \ref{lem:osc_hom}  we need
to show that ${\rm osc}_U(\widetilde{\mathcal{W}}_\psi \psi) \in
{\rm L}^1_{v'}(G)$. Let $v(x,h) \le (1+|x|)^s w(h)$, then $v'(x,h) \le (1+|x|)^s |{\rm det}(h)|^{-1/2} w(h)$. 

We fix $V = B_1(0)$ and $W = \{ h \in H : \| h-{\rm id} \|_\infty < 1/2~,~\| h^{-1} - {\rm id}\|_\infty < 1/2\}$, and $m > d+1+s$, where $s$ is the exponent of the translation part of the weight $v$. Then $U = V \times W \subset G$ is a neighborhood of the identity in $G$, and we will show finiteness of $\| {\rm osc}_U(\widetilde{\mathcal{W}}_{\psi} \psi)\|_{{\rm L}^1_v}$.

The proof proceeds in several steps. We first show the existence of a continuous function $\varphi: H \times H \to \mathbb{R}^+$,  dependent on $\psi$, such that
\begin{equation} \label{eqn:osc_1}
 \forall g,h \in H \forall x \in \mathbb{R}^d ~:~ \sup_{y \in V} \left|\widetilde{\mathcal{W}}_\psi \psi(x+hy,g) - \widetilde{\mathcal{W}}_\psi \psi(x,g) \right|  \le \varphi(h,g) \left (1+\max(0, |x|-\|h\|_\infty) \right)^{-m}
\end{equation}
For this purpose, we employ the mean value theorem, and estimate for $y \in V$
\begin{eqnarray*}
 \lefteqn{ \left|\widetilde{\mathcal{W}}_\psi \psi(x+hy,g) - \widetilde{\mathcal{W}}_\psi \psi(x,g) \right|} \\ & \le & |hy| \sum_{|\alpha|\le 1}
\sup_{z \in x+hV} \left| \partial^\alpha \left( \psi \ast (\widetilde{\pi}(0,g) \psi^*) \right) (z) \right| \\
 & \le &  \| h \|_\infty ~|\psi \ast (\widetilde{\pi}(0,g) \psi^*)|_{1,m} \sup_{z \in x+hV} (1+|z|)^{-m} \\
 & \le &  \| h \|_\infty~|\psi \ast (\widetilde{\pi}(0,g) \psi^*)|_{1,m} \left (1+\max(0, |x|-\|h\|_\infty) \right)^{-m}~,
\end{eqnarray*} where we used that $V$ is the unit ball, and the notation $\widetilde{\pi}(0,g) = |{\rm det}(g)|^{1/2}| \pi(0,g)$. 
The proof of (\ref{eqn:osc_1}) is finished when we prove that
\[
 \varphi(h,g) = \| h \|_\infty~\left|\psi \ast (\widetilde{\pi}(0,g) \psi^*)\right|_{1,m}
\] defines a continuous function on $H \times H$. Here only continuity of the last factor may not be obvious. It follows however from the continuity of $H \ni g \mapsto\widetilde{\pi}(0,g) \psi^* \in \mathcal{S}(\mathbb{R}^d)$, continuity of convolution on $\mathcal{S}(\mathbb{R}^d)$, and continuity of the Schwartz norms.

We use this estimate to prove the existence of a constant $C>0$, dependent on $\psi$, such that
\begin{equation} \label{eqn:osc_1.5}
 \int_{\mathbb{R}^d}\sup_{(y,g) \in U} \left|\widetilde{\mathcal{W}}_{\psi} \psi (x,hg) - \widetilde{\mathcal{W}}_{\psi}\psi (x+hy,hg) \right| (1+|x|)^s dx \le C   \sup_{g \in W} \varphi(h,hg) (1+\| h \|_\infty)^{m}
\end{equation}
Plugging (\ref{eqn:osc_1}) into the left-hand side of (\ref{eqn:osc_1.5}) yields
\begin{eqnarray}
\nonumber \lefteqn{ \int_{\mathbb{R}^d} \sup_{(y,g) \in U} \left|\widetilde{\mathcal{W}}_{\psi} \psi (x,hg) - \widetilde{\mathcal{W}}_{\psi}\psi (x+hy,hg) \right| (1+|x|)^s dx \le} \\
& \le & \sup_{g \in W} \varphi(h,hg) \int_{\mathbb{R}^d}   \left (1+\max(0, |x|-\|h\|_\infty) \right)^{-m} (1+|x|)^{s} dx~.
\label{zwischen}
\end{eqnarray}
Now
\begin{eqnarray*}
 \lefteqn{ \int_{\mathbb{R}^d}   \left (1+\max(0, |x|-|h|_\infty) \right)^{-m} (1+|x))^{s} dx} \\ & = &
\int_{|x|\le \| h \|_\infty} (1+|x|)^s dx + \int_{|x|>\|h \|_\infty} (1+|x|- \| h \|_\infty)^{-m} (1+|x|)^s dx \\
& \le & C_1 (1+\| h \|_\infty)^s \| h \|_\infty^d  + \int_{\mathbb{R}^d} (1+\left|~ |x|- \|h \|_\infty~ \right| )^{-m}(1+|x|)^s dx~,
\end{eqnarray*}
where $C_1$ denotes the volume of the unit ball.
Now the estimate $\left(1+\left| ~|x|- \|h\|_\infty \right| \right) (1+\|h \|_\infty) \ge (1+|x|)$ can be used to obtain
\[
\int_{\mathbb{R}^d} (1+\left| ~|x|-\|h \|_\infty ~\right|)^{-m}(1+|x|)^s dx  \le (1+\| h \|_\infty)^m \underbrace{\int_{\mathbb{R}^d} (1+|x|)^{s-m} dx}_{=: C_2}
\]
which in summary (via $m>s$) yields
\[
 (\ref{zwischen}) \le \sup_{g \in W} \varphi(h,hg) \left( C_1 (1+\| h \|_\infty)^s \| h \|_\infty^d + C_2 (1+\|h \|_\infty)^{m} \right)
 \le  C \sup_{g \in W} \varphi(h,hg) (1+\| h \|_\infty)^{m+s+d}~,
\]
 and (\ref{eqn:osc_1.5}) is proved.

We next prove that
\begin{equation} \label{eqn:osc_2}
\forall h  \in H \forall g \in W \forall y \in \mathbb{R}^d~:~ \left|\psi((hg)^{-1}y)-\psi(h^{-1}y)\right| \le \sqrt{d} |\psi|_{1,m+1} \left( 1+|h^{-1} y|/2 \right)^{-m}
\end{equation}
Indeed, by the mean value theorem, and using the definition of $W$:
\begin{eqnarray*}
 \left|\psi((hg)^{-1}y)-\psi(h^{-1}y)\right| & \le & \sqrt{d} |(hg)^{-1}y-h^{-1}y| ~|\psi|_{1,m+1} \sup_{z \in B_{|g^{-1} h^{-1}y - h^{-1} y|}(y)} (1+|z|)^{-m-1} \\
& \le & \frac{\sqrt{d}}{2}  ~|\psi|_{1,m+1} | h^{-1} y |  \left( 1+|h^{-1} y |/ 2 \right)^{-m-1} \\
& \le & \sqrt{d}  ~|\psi|_{1,m+1} \left( 1+|h^{-1} y |/ 2 \right)^{-m} ~.
\end{eqnarray*}
Here we employed the estimate $| g^{-1}h^{-1} y - h^{-1} y | \le | h^{-1} y |/2$.

We apply this observation to derive, with a positive constant $C'>0$:
\begin{equation} \label{eqn:osc_3}
 \int_{\mathbb{R}^d} \sup_{g \in W} \left| \widetilde{\mathcal{W}}_\psi \psi (x,hg) - \widetilde{\mathcal{W}}_{\psi} \psi (x,h) \right| (1+|x|)^s dx \le C' \max(1,\|h\|_\infty^m)
\end{equation}
For the proof of (\ref{eqn:osc_3}), we use the definition of the wavelet transform, as well as (\ref{eqn:osc_2}), to obtain
\begin{eqnarray*}
\lefteqn{ \int_{\mathbb{R}^d} \sup_{g \in W} \left|  \widetilde{\mathcal{W}}_\psi \psi (x,hg) - \widetilde{\mathcal{W}}_{\psi} \psi (x,h) \right| (1+|x|)^s dx }\\
& = &   \int_{\mathbb{R}^d} \sup_{g \in W}  \left| \int_{\mathbb{R}^d} \psi(x+y) \left( \psi((hg)^{-1}y)-\psi(h^{-1}y) \right)
dy \right| (1+|x|)^s dx \\
& \le &  \int_{\mathbb{R}^d} \int_{\mathbb{R}^d} |\psi(x+y)| \sqrt{d} |\psi|_{1,m+1} \left( 1+|h^{-1} y|/2 \right)^{-m} dy ~(1+|x|)^s dx \\
& = & \int_{\mathbb{R}^d} (|\psi| \ast K_h)(x) (1+|x|)^s dx \\
& = & \| ~|\psi|\ast K_h\|_{{\rm L}^1_s} \le \| \psi\|_{{\rm L}^1_s}
\| K_h \|_{{\rm L}^1_s}~,
\end{eqnarray*}
where we employed $K_h(y) =  \sqrt{d} |\psi|_{1,m+1} \left( 1+|h^{-1} y|/2 \right)^{-m}$, as well as the weighted version of Young's inequality, which holds because of the submultiplicativity of the weight $x \mapsto (1+|x|)^s$.  Using $|h^{-1}y| \ge \| h\|_\infty^{-1} |y|$, we can estimate
\begin{eqnarray*}
 \| K_h \|_{{\rm L}^1_s} & = & \sqrt{d} |\psi|_{1,m+1}\int_{\mathbb{R}^d} (1+|h^{-1}y|/2)^{-m}(1+|y|)^s dy \\
 & \le &  \sqrt{d} |\psi|_{1,m+1} \int_{\mathbb{R}^d} \max(1,2^m\|h\|_\infty^m) (1+|y|)^{s-m}dy \le C' \max(1,\|h\|_\infty^m)~.
\end{eqnarray*}

We can now finish the proof: We consider the function
\[
 \Phi: H \ni h \mapsto \int_{\mathbb{R}^d} \left({\rm osc}_U(\widetilde{\mathcal{W}}_\psi \psi) \right) (x,h) (1+|x|)^s dx~.
\] By our assumption on the weight $v$, we have
\[
\left\|  {\rm osc}_U(\widetilde{\mathcal{W}}_\psi \psi)
\right\|_{{\rm L}^1_{v'}} \le \int_H \Phi(h) |{\rm det}(h)|^{-1}
w(h) dh~.
\]

First observe that $\Phi$ is compactly supported: With the compact set $K$ from the proof of Theorem \ref{thm:ex_int_vec},
we find that ${\rm osc}_U(\widetilde{\mathcal{W}}_\psi \psi)(x,h) \not=0$ for some $x \in \mathbb{R}^d$ requires that either
$h \in K^{-1}K$, or $hg \in K^{-1} K$ for at least one $g \in W$. In any case, this can only happen for $h \in K^{-1} K W^{-1}$, which is relatively compact.
Furthermore, the triangle inequality yields the estimate
\begin{eqnarray*}
{\rm osc}_U(\widetilde{\mathcal{W}}_{\psi} \psi)(x,h) & = & \sup_{y \in V, g \in W} \left| \widetilde{\mathcal{W}}_\psi \psi(x+hy,hg)- \widetilde{\mathcal{W}}_\psi \psi (x,h) \right|  \\
& \le & \sup_{y \in V, g\in W} \left| \widetilde{\mathcal{W}}_\psi \psi (x+hy,hg) - \widetilde{\mathcal{W}}_\psi \psi(x,hg) \right| +
 \sup_{g\in W} \left| \widetilde{\mathcal{W}}_\psi \psi (x,hg) - \widetilde{\mathcal{W}}_\psi \psi(x,h) \right|
\end{eqnarray*}
Now (\ref{eqn:osc_1.5}) and (\ref{eqn:osc_3}) provide us with
\begin{eqnarray*}
 \Phi(h) \le C   \sup_{g \in W} \varphi(h,hg) (1+\| h \|_\infty)^{m} + C' \max(1,\|h\|_\infty^m)~,
\end{eqnarray*}
which is bounded on $ K^{-1} K U^{-1}$. Since in addition $w|{\rm det}|^{-1}$ is bounded on compact sets, the proof is finished.
\end{proof}

The lemma implies together with dominated convergence that
\begin{equation} \label{eqn:osc_arbitrary}
\|  {\rm osc}_U (\mathcal{W}_\psi \psi)  \|_{{\rm L}^1_v} \to 0
\end{equation} as $U$ runs through a neighborhood base of the identity.

As a consequence, we obtain the existence of atomic decompositions. First we need to define $U$-denseness and $U$-separatedness.
\begin{definition}
 Let $U \subset G$ denote a neighborhood of the identity, and $Z=(z_i)_{i \in I} \subset G$.
\begin{enumerate}
 \item[(a)] The family $(z_i)_{i \in I}$ is called {\bf $U$-dense}, if $\bigcup_{i \in I} z_i U = G$.
 \item[(b)] The family $(z_i)_{i \in I}$ is called {\bf $U$-separated}, if $z_i U \cap z_j U = \emptyset$, whenever $i \not=j$. It is called {\bf separated}, if there exists a neighborhood $U$ of unity such that it is $U$-separated. It is called {\bf relatively separated} if it is the finite union of separated families.
\end{enumerate}
\end{definition}
Note that $U$-dense, relatively separated families always exist, for every neighborhood $U$ of the identity.

We can now apply the general machinery of \cite{Gr} to derive the
existence of atomic decompositions. As a further ingredient for
this, we require a Banach sequence space $Y_d(Z)$, which is
associated to the Banach function space $Y$ according to
\cite[2.3]{Gr}:
 Pick an arbitrary compact neighborhood $W$ of the identity, and define
 \[
  \| (c_z)_{z \in Z} \|_{Y_d} = \left\| \sum_{z \in Z} |c_z| \mathbf{1}_{z W} \right\|_{Y}~.
 \]
Let $Y_d = \{ (c_z)_{z \in Z} \in \mathbb{C}^{Z}: \| (c_z)_{z \in Z} \|_{Y_d}< \infty \}$.

If $Y = {\rm L}^{p,q}_v(G)$, for some weight $v$, it can be shown that the associated coefficient space norm is equivalent to a discrete weighted $\ell^{p,q}$-norm, i.e.
\begin{equation} \label{eqn:discrete_norm_equiv}
 \left\| (c_{j,k}) \right\|_{Y_d} \asymp \left( \sum_{j \in J} |{\rm det}(h_j)|^{q/p-1}\left( \sum_{k \in K} \left( |c_{j,k}| v(h_j x_k, h_j) |{\rm det}(h_j)|^{1/p-1/q} \right)^p \right)^{q/p} \right)^{1/q}
\end{equation} with the usual modifications for $p= \infty$ and/or $q=\infty$.  This implies in particular that the finitely supported sequences are dense in $Y_d$. 
 
 The following theorem combines \cite[Theorem S]{Gr} and \cite[Theorem
T]{Gr}, which are applicable because of (\ref{eqn:osc_arbitrary}).
\begin{theorem} \label{thm:at_dec}
Let $\psi \in \mathcal{F}^{-1} C_c^\infty(\mathcal{O}) \setminus \{
0 \}$, let $Y= {\rm L}^{p,q}_w(G)$ with $1\le p,q < \infty$. Then
there exists a neighborhood $U \subset G$ of unity such that for all
$U$-dense, relatively separated families $(z_i)_{i \in I} \subset
G$, the following statements are true:
\begin{enumerate}
 \item[(a)] There is a linear bounded map $C:  Co Y \to Y_d(Z)$ with the property that, for all $f \in Co Y$,
 \[
  f = \sum_{i \in I} C(f) (z_i) \pi(z_i) \psi~,
 \] with unconditional convergence in $\| \cdot \|_{Co Y}$.
\item[(b)] Conversely, for every sequence $(c(z_i))_{i \in I} \in Y_d(Z)$, the sum
\[
 g = \sum_{i \in I} c(z_i) \pi(z_i) \psi
\] converges unconditionally in $\| \cdot \|_{Co Y}$, with $\| g \|_{Co Y} \preceq \| (c(z_i))_{i \in I} \|_{Y_d(Z)}$.
 \item[(c)] The norms $\| f \|_{Co Y}$ and $\| (\mathcal{W}_\psi f (z_i))_{i \in I}\|_{Y_d(Z)}$ are equivalent. Moreover,
$ f \in {Co Y}$ iff $ (\mathcal{W}_\psi f (z_i))_{i \in I} \in Y_d(Z)$.
\end{enumerate}
\end{theorem}

\begin{remark}
 The neighborhood $U$ in Theorem (\ref{thm:at_dec}) is chosen according to the requirement that
\[
 \| {\rm osc}_U W_{\psi} \psi \|_{L^1_{v_0}} < 1~. 
\] This shows that the same sampling density works for all coorbit spaces $Y$ for which $v_0$ is controlling. Thus, for the same range of weighted ${\rm L}^{p,q}$-spaces that we found in Remark \ref{rem:controlling_dependence}, it is possible to choose a wavelet $\psi$ and a sampling set $Z$ such that the associated wavelet system is a Banach frame in those spaces simultaneously.
\end{remark}

Repeating the proof of Theorem \ref{thm:ex_int_vec}, with the obvious modifications, allows to identify a space of Schwartz functions contained in all coorbit spaces. The density statement is a consequence of Part (a) of Theorem \ref{thm:at_dec}.
\begin{corollary} \label{cor:C_c_O_dense}
 Let $Y= {\rm L}^{p,q}_w(G)$ with $1 \le p,q \le \infty$. Then $\mathcal{F}^{-1} C_c^\infty(\mathcal{O}) \subset Co Y$. The inclusion is dense for $p,q < \infty$.
\end{corollary}

\subsection{Window dependence of coorbit spaces in the reducible case}

As already mentioned, one of the key properties provided by
coorbit space theory is {\em consistency}: The definition of ${\rm
Co}(Y)$ is independent of the choice of the analyzing wavelet $\psi
\in \mathcal{H}_{1,v}$. This property ensures that ``nice''
behaviour of the wavelet coefficients $\mathcal{W}_\psi f$ of a
function $f$ (as quantified by the coorbit space norm) can be
attributed to the analyzed function $f$, rather than to the
analyzing wavelet $\psi$. It turns out that this property crucially
depends on irreducibility of the underlying representation.  While
the results of \cite{FeiGr0,FeiGr1,FeiGr2} reflect this, and are
only formulated for irreducible representations, it is not always
easy for the reader to pinpoint where precisely the requirement is
needed. Moreover, recent investigations of wavelet inversion
formulae (such as \cite{FuMa,LaWeWeWi,Fu_LN}) and of coorbit theory
\cite{Bo,ChOl1,ChOl2} beyond the irreducible setting indicate that
irreducibility may not be essential for the development of coorbit
space theory. For instance, the authors of \cite{ChOl1} emphasize
the property that the reproducing kernel associated with the wavelet
inversion formula is integrable, as crucial for their approach,
rather than irreducibility.

This subsection gives an explicit example of a
reducible representation with associated integrable reproducing kernels, for
which consistency is lost. This result is of some independent
interest, as the above-mentioned sources from the literature
indicate: If one discards the irreducibility condition, more care and/or additional considerations are required to justify a particular choice of wavelet and the definition of the associated function spaces. In the context of this paper, the main purpose of the example is to explain why we focus on irreducible representations.

We consider the group $G = \mathbb{R}^2 \rtimes \mathbb{R}^+$, with
the multiplicative group acting diagonally, and let $\pi$ denote the
quasi-regular representation of $G$ acting on ${\rm
L}^2(\mathbb{R}^2)$. We will construct two functions
$f, g \in {\rm L}^2(G)$ with the following properties:
\begin{enumerate}
 \item[(i)] The wavelet transforms $\mathcal{W}_f,\mathcal{W}_g: {\rm L}^2(\mathbb{R}^2) \to {\rm L}^2(G)$ are both well-defined and isometric. Hence $f$ and $g$ give rise to wavelet reconstruction formulae, and the image spaces of $\mathcal{W}_f,\mathcal{W}_g$ have reproducing kernels.
 \item[(ii)] The reproducing kernels associated to each analyzing wavelet are integrable. More precisely, $\mathcal{W}_g g = \mathcal{W}_f f \in {\rm L}^1_v(G)$, for any continuous weight of the form $v(x,s) = v_0(s)$.
 \item[(iii)]  $\mathcal{W}_f g \not\in {\rm L}^1_v(G)$. In particular, the coorbit spaces $Co ({\rm L}^1_v(G))$ defined using either $\mathcal{W}_f$ or $\mathcal{W}_g$ as analyzing wavelets do not coincide.
\end{enumerate}

For this purpose, pick $0 \not= \psi_0 \in
C_c^\infty(\mathbb{R}^+)$, satisfying
\[
\int_0^\infty \frac{|\psi_0(s)|^2}{s} ds = 1~.
\]
Then, define $f \in {\rm L}^2(\mathbb{R}^2)$ via $\widehat{f}(\xi) =
\psi_0(|\xi|)$. In addition, define $g \in {\rm L}^2(\mathbb{R}^2)$
by
\[
\widehat{g}(\xi) = {\rm sign}(\langle \xi, (1,0)^T \rangle)~\cdot
\widehat{f}(\xi) .
\] Now by construction of $f$ and $g$, one has
\[
 \int_{\mathbb{R}^+} \frac{|\widehat{f}(s \xi)|^2}{s} ds =  \int_{\mathbb{R}^+} \frac{|\widehat{g}(s \xi)|^2}{s} ds =1
\] and this implies (i), by \cite[Theorem 1]{FuMa}.

The computation
\[
\mathcal{W}_f f (x,s) = |s| \left( \widehat{f} \cdot (D_{s}
\overline{\widehat{f})} \right)^{\vee}(x)~,
\] and the fact that $\widehat{f} \in C_c^\infty(\mathbb{R}^2)$
imply that indeed $\mathcal{W}_f f \in {\rm L}^1_v(G)$, by the same reasoning as in the proof of
Theorem \ref{thm:ex_int_vec}.
On the other hand, the fact that ${\rm sign}(\langle \xi, (1,0)^T
\rangle)$ is constant on all half-lines (i.e., dual orbits), implies
that
\[
\mathcal{W}_g g (x,s) = |s| \left( \widehat{g} \cdot (D_{s}
\overline{\widehat{g}} ) \right)^{\vee}(x) = |s| \left( \widehat{f}
\cdot (D_{s} \overline{\widehat{f}}) \right)^{\vee}(x) = \mathcal{W}_f f
(x,s)~.
\] Thus (ii) is established.

Finally, the same calculations yield
\[
\mathcal{W}_f g (x,s) = |s| \left( \widehat{g} D_{s} \overline{\widehat{f}}
\right)^{\vee}(x) = |s| \left( {\rm sign}(\langle \cdot, (1,0)^T
\rangle)\cdot \widehat{f} \cdot (D_{s^{-1}} \overline{\widehat{f}})
\right)^{\vee}(x)~.
\] Now, whenever the function $ \xi \mapsto {\rm sign}(\langle \xi, (1,0)^T
\rangle)\cdot \widehat{f}(\xi) \cdot (D_{s}
\overline{\widehat{f}}(\xi))$ is not identically zero, it is
discontinuous near the $y$-axis, in a way that cannot be remedied by
changing the function on a set of measure zero. This yields that
$\mathcal{W}_f g (\cdot, s) \not\in {\rm L}^1(\mathbb{R}^2)$, for
all $s \in \mathbb{R}^+$ with $\mathcal{W}_f g(\cdot,s) \not= 0$. By
property (i) (and because $g \not= 0$), this set has positive
measure. But then $\mathcal{W}_f g \not\in {\rm L}^1_v(G)$.

%
\section{Temperately embedded dual orbits and vanishing moments} \label{sect:mod_embed}

This section contains the central estimates of this paper. We start
out by defining the notion of temperately embedded dual orbits.
For this purpose, the following properties of open orbits established in \cite{Fu98} will be useful:
\begin{lemma} \label{lem:prop_open_orbit}
 Let $H$ denote an admissible matrix group, and $\mathcal{O}$ the associated open dual orbit.
 \begin{enumerate}
  \item[(a)] For all $\xi \in \mathbb{R}^d$: If $ \xi \in \mathcal{O}$, then $\mathbb{R}^+ \cdot \xi \subset \mathcal{O}$. 
  \item[(b)] There exists a polynomial $P_{\mathcal{O}}$ with real coefficients of degree $k \le 2d$ such that 
  \[
   \mathcal{O} = \{ \xi \in \mathbb{R}^d~:~P(\xi) \not= 0 \}~. 
  \]
 \end{enumerate}

\end{lemma}
\begin{prf}
 We shortly sketch the arguments for the convenience of the reader, and refer to \cite{Fu98} for more information:
Since by assumption the open dual orbit is unique, we have that $\xi \in \mathcal{O}$ holds precisely if $H^T \xi$ is an open set. Let $\mathfrak{h}$ denote the Lie algebra of $H$, then by basic results of differential geometry (such as Sard's theorem) it follows that $H^T \xi$ is open iff the map 
\[
 M_\xi : \mathfrak{h} \ni X \mapsto X^T \xi \in \mathbb{R}^d
\] has maximal rank. In particular, since $M_{r \cdot \xi} = r M_\xi$, it follows that $\mathcal{O}$ is closed under multiplication with positive scalars. 

Moreover, the rank condition holds if and only if at least one $d \times d$-subdeterminant of $M_\xi$ is nonzero. Thus summing over the squares of these subdeterminants yields a polynomial $P$ with the desired properties. 
\end{prf}

For reasons that will become apparent in Lemma \ref{lem:generate_van_moments} below, one is generally interested in finding a polynomial $P$ with the properties described in the lemma, and with rather low degree.  
We remark that the procedure to compute $P$, as sketched in the proof, may not be the simplest way to determine such a polynomial, nor will it necessarily give a polynomial of smallest possible degree. 
In most concrete cases $P$ can be easily guessed, as soon as the open dual orbit has been computed; see the examples in Section \ref{sect:exs}. 

We next define an important auxiliary function $A : \mathcal{O} \to \mathbb{R}^+$: 
Given any point $\xi \in \mathcal{O}$, let ${\rm
dist}(\xi,\mathcal{O}^c)$ denote the minimal euclidean distance of $\xi$ to
$\mathcal{O}^c$. We then let
\[
 A(\xi) = \min \left( \frac{{\rm dist}(\xi,\mathcal{O}^c)}{1+\sqrt{|\xi|^2-{\rm dist}(\xi,\mathcal{O}^c)^2}},
 \frac{1}{1+|\xi|} \right)~. 
\]
By definition, $A$ is a continuous function with $A(\cdot) \le 1$. In fact $A \in C_0(\mathcal{O})$, the space of functions on $\mathcal{O}$ vanishing at infinity. 

We let  $h^{-T}$ the transpose inverse of
$h \in H$, and use these notations to define the following central notion: 
\begin{definition} \label{defn:mod_embedded} Let $w: H \to \mathbb{R}^+$ denote a weight
function, $s \ge 0$, and $1 \le q < \infty$. $\mathcal{O}$ is called {\bf
$(s,q,w)$-temperately embedded (with index $\ell \in \mathbb{N}$)} if
the following two conditions hold, for a fixed $\xi_0 \in
\mathcal{O}$:
\begin{enumerate}
\item[(i)] The function $H \ni h \mapsto |{\rm det}(h)|^{1/2-1/q} (1+\|
h \|)^{s+d+1} w(h) A(h^T \xi_0)^\ell$ is in ${\rm L}^q(H)$.
\item[(ii)] The function $H \ni h \mapsto |{\rm det}(h)|^{-1/2-1/q} (1+\| h
\|)^{s+d+1} w(h) A(h^{-T} \xi_0)^\ell$ is in ${\rm L}^1(H)$.
\end{enumerate}
If $\mathcal{O}$ is $(s,q,w)$-temperately embedded for all $1 \le q <
\infty$ and $s \ge 0$, (with an index possibly depending on $s$ and $q$), the orbit
$\mathcal{O}$ is called {\bf $w$-temperately embedded}.
\end{definition}

Submultiplicativity of $w$, $|{\rm det}|$ and $\| \cdot \|$ imply that the conditions (i) and (ii) are independent
of the choice of $\xi_0 \in \mathcal{O}$. Furthermore, an obvious but useful
observation is that whenever the weight $w'$ is dominated by the
weight $w$, then every $w$-temperately embedded dual orbit is also
$w'$-temperately embedded. Finally, note that since $A\le 1$, it is possible to first fulfill conditions (i) and (ii) separately, with different indices $\ell_1,\ell_2$, and then take $\ell = \max(\ell_1,\ell_2)$ to guarantee them simultaneously. 

The aim of this section is to derive estimates of the form
\[
\| \mathcal{W}_\psi f \|_{{\rm L}^{p,q}_v(G)} \preceq
|\widehat{f}|_{r,m} |\widehat{\psi}|_{r,m}~,
\] whenever the orbit is temperately embedded, and $f$ and $\psi$ fulfill suitable vanishing moment conditions. These are
clarified in the next definition:
\begin{definition} \label{defn:van_mom}
 Let $r \in \mathbb{N}$ be given.
 $f \in {\rm L}^1(\mathbb{R}^d)$ {\bf has vanishing moments in $\mathcal{O}^c$ of order $r$} if all
 distributional derivatives $\partial^\alpha \widehat{f}$ with $|\alpha|\le r$ are
 continuous functions, and all derivatives of degree $|\alpha|<r$ are identically vanishing on $\mathcal{O}^c$.
\end{definition}

We will need several lemmas. The first one motivates the
introduction of our auxiliary function $A$ from above.
\begin{lemma} \label{lem:decay_est_pd}
Assume that $f$ has vanishing moments in $\mathcal{O}^c$ of order $r$, and that $m \ge r$. Then there exists
a constant $C>0$, dependent only on $r$ and $d$, such that for all $\xi \in
\mathbb{R}^d$:
\[
\left|\widehat{f}(\xi) \right| \le  C |\widehat{f}|_{r,m} A(\xi)^{r}
\]
\end{lemma}
\begin{proof}
Let $\xi \in \mathcal{O}$ be arbitrary, and $\xi' \in \mathcal{O}^c$
with $|\xi-\xi'| = {\rm dist}(\xi,\mathcal{O}^c)$; the existence of $\xi'$ is guaranteed by the usual compactness arguments. Since $\mathbb{R}^+ \cdot \xi' \subset \mathcal{O}^c$ by Lemma \ref{lem:prop_open_orbit}(a), the minimizer property of $\xi'$ entails that $\xi'$ and $\xi-\xi'$ are orthogonal, and the Pythagorean Theorem yields
\[
 |\xi'| = \sqrt{|\xi|^2 - {\rm dist}(\xi,\mathcal{O}^c)^2}
\]

Now let $\varphi :
t\mapsto \widehat{f}(\xi'+ t (\xi-\xi'))$, and consider the Taylor
expansion of $\varphi$ of order $r-1$ around $0$. By assumption, the
Taylor polynomial is zero, and thus Taylor's formula yields
\[ \widehat{f}(\xi) = \frac{\varphi^{(r)}(t)}{r!} \] for some suitable $t$ between 
$0$ and $1$. 

By induction over $r$ one obtains real constants $c_\alpha$, for all multiindices $\alpha$ with $|\alpha|=r$, such that
\[ 
\varphi^{(r)}(t) = \sum_{|\alpha|=r} c_\alpha (\xi-\xi')^\alpha (\partial^\alpha \widehat{f})(\xi'+t(\xi-\xi')) ~~~,
\] recall the notation $v^\alpha = v_1^{\alpha_1} \cdot \dots \cdot v_d^{\alpha_d}$. But this implies
\begin{eqnarray*}
 |\varphi^{(r)}(t)| & \le &  C \max_{|\alpha|=r} |\xi-\xi'|^r \left|   (\partial^\alpha \widehat{f})(\xi'+t(\xi-\xi')) \right| \\ 
& \le & C' \left| \widehat{f} \right|_{r,m} |\xi-\xi'|^r (1+|\xi'+t(\xi-\xi')|)^{-m}~~. 
\end{eqnarray*}
By the Pythagorean Theorem, $|\xi'| \le |\xi'+t(\xi-\xi')|$, and thus finally
\[
 |\widehat{f}(\xi)| \le \frac{C'}{r!}  |\widehat{f}|_{r,m} \frac{ {\rm dist}(\xi,\mathcal{O}^c)^r}{\left(1+ \sqrt{|\xi|^2 - {\rm dist}(\xi,\mathcal{O}^c)^2} \right)^m}~.
\]

On the other hand,
\[
 |\widehat{f}(\xi)| \le  |\widehat{f}|_{r,m} (1+|\xi|)^{-r}
\] is clear by definition of the Schwartz norm. 
Hence, since $m \ge r$, the lemma is proved. 
\end{proof}

In the following lemma, we use the convention $d/\infty =0$.
\begin{lemma} \label{lem:lp_vs_l1} Let $1 \le p \le \infty$. Pick an integer $t>0$ with $t >s+d/p$.
Then there exists a constant $C>0$, such that, for all $g \in {\rm L}^1(\mathbb{R}^d)$,
\[
\| g \|_{{\rm L}^p_s} \le C \max_{|\alpha| \le t} \| \partial^\alpha \widehat{g}
\|_1~.
\]
\end{lemma}
\begin{proof}
The estimate is trivial if one of the $\partial^\alpha \widehat{g}$
is not integrable. If all of these functions are integrable, we
obtain that $\mathcal{F}^{-1}(\partial^\alpha \widehat{g})(x) = (-2
\pi i x)^{\alpha} g (x)$, and $g$ is bounded. This implies
\[
|g(x)| \preceq (1+|x|)^{-t}  \max_{|\alpha| \le t} \|
\partial^\alpha \widehat{g} \|_1~.
\] and the estimate follows from integrability of $(1+|x|)^{sp-tp}$ since
$(t-s)p \ge t-s >d$.
\end{proof}

For the following, we remind the reader of $D_h \widehat{g}: \xi \mapsto
\widehat{g}(h^T \xi)$. The lemma exhibits the role of the auxiliary functions: They serve as {\em envelopes} for $\widehat{\psi}$ and its derivatives, as soon as $\psi$ exhibits sufficient smoothness and vanishing moments. This role will become particularly apparent in the proof of Theorem \ref{thm:central} below. 
\begin{lemma} \label{lem:decay_est_prod}
Assume that $f,\psi \in {\rm L}^1(\mathbb{R}^d)$ have vanishing
moments of order $r$ in $\mathcal{O}^c$, and fulfill
$|\widehat{f}|_{r,r-|\alpha|} < \infty, |\widehat{\psi}|_{r,r-|\alpha|} < \infty$, for some
multiindex $\alpha$ with $|\alpha| \le r$. Then there exists a
constant $C>0$, independent of $f$ and $\psi$, such that
\[
 | \partial^\alpha (\widehat{f} \cdot D_h \widehat{\psi})(\xi)| \le
 C |\widehat{f}|_{r,r-|\alpha|} |\widehat{\psi}|_{r,r-|\alpha|}  
(1+\| h \|_\infty)^{|\alpha|}
 A(\xi)^{r-|\alpha|} A(h^T \xi)^{r-|\alpha|}
\]
\end{lemma}
\begin{proof}
Using the Leibniz formula, we obtain
\[
\left|\partial^\alpha (\widehat{f} \cdot D_h \widehat{\psi}) (\xi)
\right| \le \sum_{\gamma+\beta =\alpha}  \frac{\alpha!}{\beta! \gamma!}  \left|\left(
\partial^\beta \widehat{f} \right)(\xi) \right|~ \left| \left(
\partial^\gamma (D_h \widehat{\psi}) \right)(\xi) \right| ~,
\] where we used the notation $\alpha! = \prod_i \alpha_i !$. 
We next use the chain rule to obtain inductively for any multiindex $\gamma$
\[
 |\partial^\gamma (D_h \widehat{\psi})(\xi)| \preceq (1+\|h\|_\infty)^{|\gamma|} \sum_{|\gamma'|\le |\gamma|}
 \left|\left(\partial^{\gamma'} \widehat{\psi} \right) (h^T \xi)\right|~. 
\]
Thus we obtain
\[
 \left|\partial^\alpha (\widehat{f} \cdot D_h \widehat{\psi}) (\xi)
\right| \preceq (1+\|h \|_\infty)^{|\alpha|} \sum_{|\beta|, |\gamma| \le |\alpha|}  
 \left|\left(
\partial^\beta \widehat{f} \right)(\xi) \right| \cdot \left| \left( \partial^\gamma \widehat{\psi} \right) (h^T \xi) \right|~.
\]

For any $\gamma,\beta$ occurring in the sum, $(\partial^\beta
\widehat{f})^\vee$ and $(\partial^\gamma \widehat{\psi})^\vee$  have vanishing moments in $\mathcal{O}^c$ of order $ \ge r-|\alpha|$. Hence, Lemma \ref{lem:decay_est_pd}
applied to both factors yields the desired estimate. 
\end{proof}

We now come to one of the central estimates of this paper. Note that, when applying this lemma to weights of the type $v(x,h) = (1+|x|+\| h \|_\infty)^s w(h)$, the weight $w$ in the theorem needs to be replaced by $w_s$. 
\begin{theorem} \label{thm:central}
Let $1 \le p \le \infty$ and $1 \le q < \infty$. Let $v$ be a weight on $G$ satisfying $v(x,h) \le (1+|x|)^s w(h)$, for a suitable weight $w$ on $H$.  
Assume that $\mathcal{O}$ is $(s,q,w)$-temperately embedded with index
$\ell$. Let $f, \psi \in {\rm L}^1(\mathbb{R}^d)$ have vanishing
moments in $\mathcal{O}^c$ of order $t > \ell+s+1+d/p$. Then there
exist constants $C_{p,q} >0$, independent
of $f,\psi$, such that
\[
\| W_{\psi} f \|_{{\rm L}^{p,q}_v} \le C_{p,q} |\widehat{f}|_{t,t}
|\widehat{\psi}|_{t,t}~.
\]
\end{theorem}
\begin{proof}
 The Haar measure on $G$ is
given by $d(x,h) = dx \frac{dh}{|\det(h)|}$. Our conventions
regarding the operator $D_h$ allow to rewrite
\[
\mathcal{W}_{\psi} f (x,h) = \left( |{\rm det}(h)|^{1/2} \widehat{f}
\cdot (D_h \overline{\widehat{\psi}}) \right)^{\vee}(x)~.
\]
Thus, by our assumptions on the weights $v,w$:
\[
\left\| \mathcal{W}_\psi f \right\|_{{\rm L}^{p,q}_v}^q \le \int_{H}
|{\rm det}(h)|^{q/2} \left\| \left( \widehat{f} \cdot (D_h
\overline{\widehat{\psi}}) \right)^\vee \right\|_{{\rm L}^p_s}^q w(h)^q
\frac{dh}{|{\rm det}(h)|}~.
\] By Lemma \ref{lem:lp_vs_l1}, we have
\[
 \left\| \left( \widehat{f} \cdot (D_h
\overline{\widehat{\psi}}) \right)^\vee \right\|_{L^p_s} \preceq
\sum_{|\alpha| \le s+1+d/p} \left\| \partial^\alpha \left( \widehat{f}
\cdot \left( D_h \overline{\widehat{\psi}} \right) \right)
\right\|_1~,
\] with a constant depending on $p$, but not on $\psi$ and $f$. Thus, by the ${\rm L}^q$-triangle inequality
\[
\left\| \mathcal{W}_\psi f \right\|_{{\rm L}^{p,q}_v} \preceq
\sum_{|\alpha| \le s+1+d/p} \left( \int_H |{\rm det}(h)|^{q/2}\left\|
\partial^\alpha \left( \widehat{f} \cdot \left( D_h
\overline{\widehat{\psi}} \right) \right) \right\|_1^q w(h)^q
\frac{dh}{|{\rm det}(h)|} \right)^{1/q}~,
\] and we will estimate the integral for each $|\alpha| \le s+1+d/p$  independently. By Lemma
\ref{lem:decay_est_prod}, we find
\begin{eqnarray*}
\lefteqn{ \int_H |{\rm det}(h)|^{q/2}\left\|
\partial^\alpha \left( \widehat{f} \cdot \left( D_h
\overline{\widehat{\psi}} \right) \right) \right\|_1^q w(h)^q
\frac{dh}{|{\rm det}(h)|} \preceq} \\ & \preceq & \int_H \left(
\int_{\mathbb{R}^d} |\widehat{f}|_{t,t} |\widehat{\psi}|_{t,t} |{\rm
det}(h)|^{1/2} w(h) (1+\| h \|)^{|\alpha|}
 A(\xi)^{t-|\alpha|} A(h^T \xi)^{t-|\alpha|} d\xi \right)^q
 \frac{dh}{|{\rm det}(h)|}~.
\end{eqnarray*}
Thus, by Minkowski's inequality for integrals,
\begin{eqnarray*}
\lefteqn{ \left(  \int_H |{\rm det}(h)|^{q/2}\left\|
\partial^\alpha \left( \widehat{f} \cdot \left( D_h
\overline{\widehat{\psi}} \right) \right) \right\|_1^q w(h)^q
\frac{dh}{|{\rm det}(h)|} \right)^{1/q} \preceq}\\ & \preceq &
|\widehat{f}|_{t,t} |\widehat{\psi}|_{t,t} \int_{\mathbb{R}^d} \left(
\int_{H} |{\rm det}(h)|^{q/2} w(h)^q  (1+\| h
\|_\infty)^{q|\alpha|} A(\xi)^{q(t-|\alpha|)} A(h^T \xi)^{q(t-|\alpha|)}
\frac{dh}{|{\rm det}(h)|} \right)^{1/q} d\xi ~,
\end{eqnarray*}
and it remains to prove that the integral is finite, for all
$|\alpha| \le s+1+d/p$. For this purpose observe that $\mathcal{O}
\subset \mathbb{R}^d$ is of full measure, hence it suffices to
integrate over $\mathcal{O}$. We fix $\xi_0 \in \mathcal{O}$, then
the stabilizer in $H$ associated to $\xi_0$ is compact. Let
$p_{\xi_0}: H \to \mathcal{O}$, $p_{\xi_0}(h) = h^T \xi_0$ denote
the canonical quotient map. Then the measure $\mu_{\mathcal{O}}$,
defined by $\mu_{\mathcal{O}}(A) = \mu_H(p_{\xi_0}^{-1}(A))$ is a
well-defined Radon measure on $\mathcal{O}$, which is
Lebesgue-absolutely continuous with Radon-Nikodym derivative
\[ \frac{d\mu_{\mathcal{O}}(h^T \xi_0)}{d\lambda(h^T \xi_0)} = c_0 \frac{\Delta_H(h)}{ |{\rm
det}(h)|} ~,\] for a positive constant $c_0$; see \cite{Fu96,FuDiss} for more
details.  Hence for any
Borel-measurable $F: \mathcal{O} \to \mathbb{R}^+$,
\[
 \int_{\mathcal{O}} F(\xi) d\xi = \frac{1}{c_0} \int_H F(h^T \xi_0) \frac{|{\rm
 det}(h)|}{
 |\Delta_H(h)|} dh~.
\]
Applying this to the above integrand yields, after the substitution
$\xi = h_1^T \xi_0$, that
\begin{eqnarray*}
\lefteqn{ \int_{\mathbb{R}^d} \left(\int_H  |{\rm det}(h)|^{q/2}
w(h)^q (1+\| h \|_\infty)^{q|\alpha|} A(\xi)^{q(t-|\alpha|)} A(h^T
\xi)^{q(t-|\alpha|)} \frac{dh}{|{\rm det}(h)|} \right)^{1/q} d\xi =}
\\ & = & \frac{1}{c_0} \int_H \left( \int_{H} |{\rm det}(h)|^{q/2}
w(h)^q (1+\| h \|_\infty)^{q|\alpha|}  A((h_1h)^T \xi_0)^{q(t-|\alpha|)}
\frac{dh}{|{\rm det}(h)|} \right)^{1/q} \\ & &  A(h_1^T
\xi_0)^{t-|\alpha|}\frac{|{\rm det}(h_1)|}{\Delta_H(h_1)} dh_1
\\ & = & \frac{1}{c_0} \int_H \left(  \int_{H} |{\rm
det}(h_1^{-1}h)|^{q/2} w(h_1^{-1}h)^q (1+\| h_1^{-1} h
\|)^{q|\alpha|} A(h^T \xi_0)^{q(t-|\alpha|)} \frac{dh}{|{\rm
det}(h_1^{-1}h)|} \right)^{1/q} \\ & &  A(h_1^T
\xi_0)^{t-|\alpha|}\frac{|{\rm det}(h_1)|}{\Delta_H(h_1)} dh_1 ~,
\end{eqnarray*}
where we used left-invariance of $\mu_H$. Now applying the
submultiplicativity of $|{\rm det}|,w$ and $\|\cdot \|_\infty$  yields
\begin{eqnarray*} ... & \preceq & \int_H \left(  \int_{H} |{\rm
det}(h)|^{q/2} w(h)^q (1+\| h \|_\infty)^{q|\alpha|} A(h^T
\xi_0)^{q(t-|\alpha|)} \frac{dh}{|{\rm det}(h)|} \right)^{1/q} \\ &
&  w(h_1^{-1}) |{\rm det}(h_1)|^{1/2+1/q} (1+\|
h_1^{-1}\|_\infty)^{|\alpha|}
A(h_1^T \xi_0)^{t-|\alpha|} \frac{dh_1}{\Delta_H(h_1)} \\
& = & \left(  \int_{H} |{\rm det}(h)|^{q/2} w(h)^q (1+\| h
\|_\infty)^{q|\alpha|} A(h^T \xi_0)^{q(t-|\alpha|)} \frac{dh}{|{\rm
det}(h)|} \right)^{1/q} \times \\ & & \int_H w(h_1) |{\rm
det}(h_1)|^{-1/q-1/2} (1+\| h_1\|_\infty)^{|\alpha|} A(h_1^{-T}
\xi_0)^{t-|\alpha|} dh_1~,
\end{eqnarray*}
where the last equation is obtained by pulling the inner integral
--which is now independent of  $h_1$-- out, and inverting the
integration variable $h_1$,  cancelling $\Delta_H(h_1)$ in the
process. Now the facts that $t-|\alpha| \ge t-s-d/p-1 > \ell$, $A
\le 1$ and $|\alpha| \le s+d+1$ imply that the integrals are finite
whenever $\mathcal{O}$ is $(s,q,w)$-temperately embedded with index
$\ell$. The proof is therefore finished.
\end{proof}

\section{Vanishing moment conditions and norm estimates for coorbit
spaces} \label{sect:coorbit}

We can now easily derive sufficient criteria for the existence of
wavelets with nice properties, and also for functions to be in a coorbit space. For the
explicit construction of nice wavelets, the following lemma is quite
useful. It will allow to adapt the following simple procedure to construct univariate functions with prescribed vanishing moments (i.e., wavelets) to higher dimensions: Pick a function $\rho$ with sufficient smoothness and decay (including the derivatives), and let $\psi = \rho^{(n)}$. Then $\psi$ will have vanishing moments of order $n$. 
\begin{lemma} \label{lem:generate_van_moments}
There exists a linear partial differential operator $D$ with
constant coefficients and degree $k \le 2d$ with the following property:
If $f$ and all its partial derivatives of order $\le sk$ is  in ${\rm
L}^1(\mathbb{R}^d)$, then $D^s f$ has vanishing moments of order
$s$ in $\mathcal{O}^c$.
\end{lemma}
\begin{proof}
Let $P(\xi) = \sum_{|\alpha| \le 2d} a_\alpha \xi^\alpha$ denote the polynomial from Lemma \ref{lem:prop_open_orbit}, i.e. $P$ vanishes precisely on $\mathcal{O}^c$. Define 
\[
 D = \sum_{|\alpha| \le 2d} a_\alpha (-2 \pi i)^{-|\alpha|} \partial^\alpha~.
\] Then $D$ is a linear partial differential operator of degree $k \le 2d$, and we have
$(D^sf)^\wedge(\xi) = P(\xi)^s \cdot \widehat{f}(\xi)$. But then the Leibniz rule implies that $D^sf$
 has vanishing moments in $\mathcal{O}^c$ of order $s$. 
\end{proof}

Note that under the assumptions of the lemma, the vanishing moment condition can be expressed by
\[
\forall \xi \in \mathcal{O}^c ~ \forall \alpha \in \mathbb{N}_0^d \mbox{ with } |\alpha|<s~:~ \int_{\mathbb{R}^d} x^\alpha f(x) e^{-2 \pi i \langle \xi, x \rangle} dx = 0~.
\]

\begin{remark}
 In the following statements, the condition $|\widehat{\psi}|_{t,t}< \infty$ occurs repeatedly. It should be noted that this condition is easy to guarantee by explicit assumptions on $\psi$: Indeed, $|\widehat{\psi}|_{t,t} < \infty$ is equivalent to boundedness of $\xi \mapsto \xi^\alpha \partial^\beta \widehat{\psi}$, for all multiindices $\alpha,\beta$ of length $\le t$. But the latter condition is guaranteed by integrability of $\partial^\alpha \left( x^\beta \psi \right)$, for all such multiindices, or equivalently, by integrability of $x^\beta \partial^\alpha \psi$. It is thus fulfilled by all compactly supported $\psi$ with continuous partial derivatives up to order $t$. 
\end{remark}

We can now easily state sufficient conditions for admissibility:
\begin{corollary} \label{cor:adm_cond_vm}
Assume that $\mathcal{O}$ is $(0,2,w)$-temperately embedded with index
$\ell$, and constant weight $w\equiv 1$. Then any function $\psi \in
{\rm L}^1(\mathbb{R}^d)$ with
vanishing moments in $\mathcal{O}^c$ of order $t>\ell+d/2+1$ and
$|\widehat{\psi}|_{t,t} < \infty$ is
admissible. There exist admissible vectors $\psi \in
C_c^{\infty}(\mathbb{R}^d)$.
\end{corollary}
\begin{proof}
The first statement is a special case of Theorem \ref{thm:central}.
The existence statement follows by picking an arbitrary nonzero $f
\in C_c^\infty(\mathbb{R}^d)$ and letting $\psi = D^{\ell+d+1} f$.
By construction of the operator $D^{\ell+d+1}$,
\[ \widehat{\psi}(\xi) = P(\xi) \cdot
\widehat{f}(\xi) ~,\]  and $P$ is a polynomial vanishing precisely
on $\mathcal{O}^c$, with order at least $\ell+d+1$. Since
$\widehat{f}$ is continuous and nonzero, $\widehat{f} \cdot P
\not=0$, and thus $\psi$ is nonzero. Since in addition $V_{\psi}
\psi \in {\rm L}^2(G)$, this implies that $\psi$ is admissible.
\end{proof}

The same argument, with slight adjustments, yields sufficient
conditions for ${\rm L}^1$-admissible vectors.
\begin{corollary} \label{cor:L1_adm_cond_vm}
Let $v$ be a weight on $G$ satisfying $v(x,h) \le (1+|x|)^s w(h)$. 
Assume that $\mathcal{O}$ is $(s,1,w)$-temperately embedded with index
$\ell$. Then any function $\psi \in {\rm L}^1(\mathbb{R}^d) $ with vanishing moments in
$\mathcal{O}^c$ of order $t> \ell+s + d+1$ and $|\widehat{\psi}|_{t,t} < \infty$ is in $\mathcal{H}_{1,v}$. There exist
nonzero vectors $\psi \in C_c^{\infty}(\mathbb{R}^d) \cap
\mathcal{H}_{1,v}$.
\end{corollary}

We next turn to coorbit spaces. Here a direct application of Theorem
\ref{thm:central} yields the following sufficient condition. Note
that for the wavelet, we may take any element in
$\mathcal{F}^{-1}C_c^\infty(\mathcal{O})$, which automatically
fulfills any vanishing moment condition.
\begin{corollary} \label{cor:coorb_vm} Let $v$ be a weight on $G$ satisfying $v(x,h) \le (1+|x|)^s w(h)$. 
 Let $1 \le p \le \infty$ and $1 \le q < \infty$.
Assume that $\mathcal{O}$ is $(s,q,w)$-temperately embedded with index
$\ell$. Then any function $f \in {\rm L}^1(\mathbb{R}^d)$ with vanishing moments in
$\mathcal{O}^c$ of order $t>\ell+s+d/p+1$ and $|\widehat{\psi}|_{t,t} < \infty$ is in ${\rm
Co}(L_{w}^{p,q}(G))$.
\end{corollary}

For the systematic study of coorbit spaces, it is useful to
introduce specific spaces of test functions and tempered
distributions. We let
\[
\mathcal{S}'_{\mathcal{O}^c} = \{ \varphi \in
\mathcal{S}'(\mathbb{R}^d) ~:~ {\rm supp}(\varphi) \subset
\mathcal{O}^c \}
\]
and
\[
\mathcal{S}(\mathcal{O}) = \{ f \in \mathcal{S}(\mathbb{R}^d)~:~
\forall \varphi \in \mathcal{S}'_{\mathcal{O}^c}: \varphi(f) = 0
\}~.
\]
As intersection of kernels of tempered distributions,
$\mathcal{S}(\mathcal{O})$ is a closed subspace of
$\mathcal{S}(\mathbb{R}^d)$, thus a Fr\'echet space in its own
right, containing $\mathcal{C}_c^\infty(\mathcal{O})$. Moreover, it
is straightforward to check that the dual
$\mathcal{S}(\mathcal{O})'$ can be identified with the quotient
space $\mathcal{S}'(\mathbb{R}^d)/\mathcal{S}'_{\mathcal{O}^c}$.
This should be compared to the case $\mathcal{O}^c= \{ 0 \}$ of the
similitude group: Here $\mathcal{S}'_{\{ 0 \}}$ is just the span of
the partial derivatives of the $\delta$-distribution at zero. The
inverse image under the Fourier transform is the space of polynomials, and it is
customary to regard elements of the homogeneous Besov spaces as
tempered distributions modulo polynomials.

Clearly any element
in $\mathcal{F}^{-1} \mathcal{S}(\mathcal{O})$ has vanishing moments
in $\mathcal{O}^c$ of arbitrary order.
Thus, as a consequence of Corollary \ref{cor:coorb_vm}, we obtain
the following embedding result. Note that the continuity statement
follows from the fact that the estimate in Theorem \ref{thm:central}
uses the Schwartz norms. The density statement follows from Corollary \ref{cor:C_c_O_dense}.
\begin{corollary}  Let $1 \le p \le \infty$ and $1 \le q < \infty$.
Assume that $\mathcal{O}$ is $(s,q,w)$-temperately embedded. Then
$\mathcal{F}^{-1}\mathcal{S}(\mathcal{O}) \subset {\rm Co}({\rm
L}^{p,q}_v(G))$ is a continuous embedding, and dense for $p < \infty$.
\end{corollary}

As a specific consequence, we note that by duality we obtain the
continuous embedding $\mathcal{H}_{1,v}^{\sim} \subset
\mathcal{F}^{-1}\mathcal{S}'(\mathbb{R}^d)/\mathcal{F}^{-1}\mathcal{S}'_{\mathcal{O}^c}$.
This has the nice consequence that all coorbit spaces for which the respective temperate embeddedness conditions hold fit in the following framework: We pick a wavelet $\psi \in \mathcal{F}^{-1} C_c^\infty(\mathcal{O}) \subset \mathcal{S}(\mathbb{R}^d)$, and define the wavelet
transform of an arbitrary tempered distributions $\varphi$ by letting \[ \mathcal{W}_\psi
\varphi (x,h) = \langle \varphi, \pi(x,h) \psi \rangle~, \] where
the bracket $\langle \cdot, \cdot \rangle :
\mathcal{S}'(\mathbb{R}^d) \times \mathcal{S}(\mathbb{R}^d) \to
\mathbb{C}$ denotes the sesqui-linear version of the duality between
$\mathcal{S}$ and $\mathcal{S}'$. This definition then extends the
definition of the wavelet transform on $\mathcal{H}_{1,v}^\sim$, and 
we can alternatively define ${\rm Co} Y$ as the space of all tempered distributions $\varphi$ with $\mathcal{W}_\psi \varphi \in Y$, identifying two tempered distributions $\varphi_1, \varphi_2$ whenever 
$(\varphi_1-\varphi_2)^\wedge \in \mathcal{S}'_{\mathcal{O}^c}$ (in which case $\mathcal{W}_\psi (\varphi_1-\varphi_2)$ vanishes identically). 

Thus every coorbit space arises as a subspace
of a quotient space of $\mathcal{S}'(\mathbb{R}^d)$. An alternative could consist in
using the dense embedding $\mathcal{F}^{-1}C_c^\infty(\mathcal{O})
\subset \mathcal{H}_{1,v}$, which is available without any
assumptions on the dual orbit. However, this space is less
compatible with the Fourier transform; the right-hand side of the dual embedding $\mathcal{H}_{1,v}^\sim \subset \left( \mathcal{F}^{-1} C_c^\infty(\mathcal{O}) \right)'$ does not seem particularly accessible. 

\section{Temperately embedded orbits in dimension 2}
\label{sect:exs}

In this section we will exhibit a large variety of examples for
which the rather technical conditions of the previous sections can
be verified. The following results will show that, for {\em any}
admissible dilation group in dimension two and a large class of
weights $w$ (which include all weights that have been studied
previously in this context), the dual orbit is $w$-temperately
embedded, which makes Theorem \ref{thm:central} and its various
corollaries available for all these groups. Essentially, we will
treat three different types of groups, which constitute a set of
representatives of all admissible matrix groups in dimension 2, up
to conjugacy and finite extensions. The fact that the list of groups
treated in the following subsections is such a set of
representatives is proved in \cite{FuDiss,FuCuba}. Furthermore, note
that the properties that we are interested in are well-behaved under
conjugation: If $H$ and $H'$ are conjugate by some $g \in {\rm
GL}(\mathbb{R},d)$, then weights on $H$ correspond to their
conjugate images on $H'$, and the associated dual orbits are mapped
into each other by $g^T$, and thus inherit the property of being
temperately embedded. Similarly, if $H< H'$ are two admissible
dilation group, then $H$ has finite index in $H'$, with the same
dual open orbit \cite{FuCuba}, and all relevant properties are transferred from $H$ to
$H'$ and vice versa.

\subsection{The similitude group}
Let
\[
 H = \left\{ \left( \begin{array}{cc} a & b \\ -b & a \end{array}
 \right)~:~a, b \in \mathbb{R}, a^2+b^2>0 \right\}~.
\] This is the {\bf similitude group} in two dimensions, and it is
the first systematically studied higher-dimensional dilation group,
\cite{Mu,AnMuVa}. The coorbit spaces associated to the isotropic weights on $H$ that we shall employ below
are just the usual (isotropic) homogeneous Besov spaces.

The dual orbit is computed as $\mathcal{O} =
 \mathbb{R}^2 \setminus \{ 0 \}$.
Let $h = \left( \begin{array}{cc} a & b \\ -b & a \end{array}
 \right) \in H$ be given.
 We fix $\xi_0 = (1,0)^T \in \mathcal{O}$, then $h^T \xi_0 = (a,b)$.
We assume that the weight $w$ on $H$ fulfills $w(h) \le
(a^2+b^2)^u+(a^2+b^2)^{-u}$ with $u>0$. For the matrix norm, the choice $\| h \| =
\left( a^2+b^2 \right)^{1/2}$ will be most convenient, and we will use the
euclidean norm on $\mathbb{R}^2$. Haar measure on
$H$ is given by $\frac{dadb}{a^2+b^2}$.

Given any $\xi \not= (0,0)^T$, it is obvious that ${\rm
dist}(\xi, \mathcal{O}^c) = |\xi|$, and hence
\[
 A(\xi) = \min \left( |\xi|,\frac{1}{1+|\xi|} \right)~.
\]
Thus, in order to verify condition (i) of Definition
\ref{defn:mod_embedded}, we need to prove finiteness of
\[
\int_{\mathbb{R}^2 \setminus \{ 0 \}} (a^2+b^2)^{q/2-1}
(1+(a^2+b^2)^{1/2})^{(s+3)q} (a^2+b^2)^{\pm uq} ~\min  \left(
(a^2+b^2)^{1/2},\frac{1}{1+(a^2+b^2)^{1/2}} \right)^{\ell
q}~\frac{dadb}{a^2+b^2}~,
\]
for $\ell$ sufficiently large. But this follows easily from the
observation that
\[
(a^2+b^2)^{\pm 1/2} \cdot \min  \left(
(a^2+b^2)^{1/2},\frac{1}{1+(a^2+b^2)^{1/2}} \right) \le 1,
\] and the integrability of $(1+(a^2+b^2)^{1/2})^{-r}$ whenever $r>2$.

For the second condition, we first compute
\[
A(h^{-T} \xi_0) = \min  \left(
(a^2+b^2)^{-1/2},\frac{1}{1+(a^2+b^2)^{-1/2}} \right)~,
\] hence we need to verify finiteness of
\[
\int_{\mathbb{R}^2 \setminus \{ 0 \}} (a^2+b^2)^{-1/2-1/q}
(1+(a^2+b^2)^{1/2})^{s+3} (a^2+b^2)^{\pm u} ~\min  \left(
(a^2+b^2)^{-1/2},\frac{1}{1+(a^2+b^2)^{-1/2}} \right)^{\ell
}~\frac{dadb}{a^2+b^2}~.
\] This follows by a similar argument.

\begin{remark}
 The calculations carried out for $d=2$ immediately generalize to similitude groups in higher dimensions.
\end{remark}

\begin{remark} \label{rem:suff_mom_cond_L1}
 Note that in order to apply Corollary \ref{cor:L1_adm_cond_vm} to obtain analyzing vectors for the coorbit spaces, we need sufficient conditions for ${\rm L}^1_{v_0}(G)$, where $v_0$ is a control weight for ${\rm L}^{p,q}_v$. By \ref{lem:weight_control} (and the subsequent remark), an upper bound for $v_0$ is obtained by multiplying $v$ with certain powers of $(1+\| h^{\pm 1} \|_\infty)$,  $(|{\rm det}(h)|+|{\rm det}(h)|^{-1})$, and $(1+\Delta_G(0,h^{\pm 1}))$. Each of these additional factors is controlled by suitable powers of the auxiliary function $A$, and thus Corollary \ref{cor:L1_adm_cond_vm} becomes applicable. 
 
A polynomial $P$ vanishing precisely on $\mathcal{O}^c$ is given by $P(\xi) = \xi_1^2+\xi_2^2$.  Hence we can take the Laplacian as the differential operator $D$ inducing vanishing moments of a given order via Lemma \ref{lem:generate_van_moments}.
\end{remark}

\subsection{The diagonal group}
Let
\[
 H = \left\{ \left( \begin{array}{cc} a & 0 \\ 0 & b \end{array}
 \right)~:~a, b \in \mathbb{R}, ab\not= 0 \right\}~.
\] Haar measure on
this dilation group is given by $\frac{dadb}{|ab|}$. The dual orbit
is determined as
\[ \mathcal{O} =
 \mathbb{R}^2 \setminus \left( (\{ 0 \} \times \mathbb{R}) \cup (\mathbb{R} \times \{ 0
\})\right)~.\] For $h = \left( \begin{array}{cc} a & 0 \\ 0 & b
\end{array}
 \right) \in H$ and $\xi_0 = (1,1)^T \in \mathcal{O}$, we obtain $h^T \xi_0 = (a,b)^T$.
We assume that the weight $w$ on $H$ fulfills $w(h) \le
(|a|+|a|^{-1})^t(|b|+|b|^{-1})^u$, with $t,u>0$. Most other reasonable weights
(isotropic or anisotropic) can be majorized by this weight, for
suitably large $u$. The matrix norm is defined as $\| h \| =
\max(|a|,|b|)$. The determinant of $h$ is $ab$.

For any $\xi = (\xi_1,\xi_2)^T \in \mathcal{O}$, one computes that
\[ {\rm dist}(\xi, \mathcal{O}^c) = \min(|\xi_1|,|\xi_2|) ~\mbox{ and }~
\sqrt{|\xi|^2 - {\rm dist}(\xi, \mathcal{O}^c)^2} =  \max(|\xi_1|,|\xi_2|)~. \] Hence we obtain 
\[
 A(\xi) = \min \left( \frac{\min(|\xi_1|,|\xi_2|)}{1+\max(|\xi_1|,|\xi_2|)},\frac{1}{1+|\xi_1|+|\xi_2|} \right)~,
\]
where we have taken the liberty to replace the euclidean norm by the slightly more tractable $\ell^1$-norm.  
Thus, condition (i) of Definition \ref{defn:mod_embedded} requires
finiteness of
\begin{eqnarray*}
 \lefteqn{\int_{\mathbb{R'}} \int_{\mathbb{R}'} |ab|^{q/2-1}
(1+\max(|a|,|b|))^{(s+3)q} (|a|+|a|^{-1})^{qt}(|b|+|b|^{-1})^{uq}} \\
 &\times& \min
\left( \frac{\min(|a|,|b|)}{1+\max(|a|,|b|)},\frac{1}{1+|a|+|b|}
\right)^{\ell q}~\frac{dadb}{|ab|}~,
\end{eqnarray*} for $\ell$ sufficiently large. In order to prove
this, we first observe that
\[
|a|^{\pm 1} ~\min \left(
\frac{\min(|a|,|b|)}{1+\max(|a|,|b|)},\frac{1}{1+|a|+|b|} \right)
\le 1
\] holds, and likewise for $b$. Thus the product
\[
|ab|^{q/2-2} (1+\max(|a|,|b|))^{(s+3)q}
(|a|+|a|^{-1})^{qt}(|b|+|b|^{-1})^{qu} ~\min \left(
\frac{\min(|a|,|b|)}{1+\max(|a|,|b|)},\frac{1}{1+|a|+|b|}
\right)^{r}
\] is bounded for $r$ sufficiently large, and then increasing $r$ by three results in a convergent integral.

For condition $(ii)$ of Definition \ref{defn:mod_embedded} we need
to ensure finiteness of
\begin{eqnarray*}
\lefteqn{\int_{\mathbb{R}'} \int_{\mathbb{R}'} |ab|^{-1/2-1/q}
(1+\max(|a|,|b|))^{s+3} (|a|+|a|^{-1})^{t}(|b|+|b|^{-1})^{u}} \\
 & \times & ~\min
\left(
\frac{\min(|a|^{-1},|b|^{-1})}{1+\max(|a|^{-1},|b|^{-1})},\frac{1}{1+|a|^{-1}+|b|^{-1}}
\right)^{\ell}~\frac{dadb}{|ab|}~.
\end{eqnarray*} Noting that
\[
\frac{\min(|a|^{-1},|b|^{-1})}{1+\max(|a|^{-1},|b|^{-1})}
 = \frac{1}{\max(|a|,|b|)}\frac{1}{1+\max(|a|^{-1},|b|^{-1})}~,
\] we can again estimate
\[
|a|^{\pm 1}
\min \left( \frac{\min(|a|^{-1},|b|^{-1})}{1+\max(|a|^{-1},|b|^{-1})},\frac{1}{1+|a|^{-1}+|b|^{-1}} \right)
\le 1~,
\] and likewise for $b$. Hence we can choose an exponent $\ell>0$ sufficiently large
such that the integrand becomes bounded, and increasing the exponent
by three results in a convergent integral.

\begin{remark}
 Also for the diagonal group, the reasoning easily extends to arbitrary dimensions.

Again, the proof showed that functions of the type $(1+\| h^{\pm 1} \|)$,  $(|{\rm det}(h)|+|{\rm det}(h)|^{-1})$ and $(1+\Delta_G(0,h^{\pm 1}))$ can be controlled by suitable powers of the auxiliary function $A$. Hence, the same reasoning as for the diagonal group allows to conclude that  Corollary \ref{cor:L1_adm_cond_vm} also provides sufficient criteria for analyzing vectors. 

A polynomial $P$ vanishing precisely on $\mathcal{O}^c$ is given by $P(\xi) = \xi_1 \xi_2$.  Hence we can take the operator $Df = \partial_{x_1} \partial_{x_2} f$ and its powers to induce vanishing moments of a given order via Lemma \ref{lem:generate_van_moments}.
\end{remark}

\subsection{The shearlet groups}
Fix a real parameter $c$, and let
\[
 H = H_c = \left\{ \left( \begin{array}{cc} a & b \\ 0 & a^c \end{array}
 \right)~:~a, b \in \mathbb{R}, a\not= 0 \right\}~.
\] Here we use the convention $a^c = {\rm sign}(a) |a|^c$ for $a < 0$. 
Note that for $c=1/2$, this is the shearlet group introduced in
\cite{KuLa}, and studied (e.g.) in \cite{DaKuStTe,DaStTe10}. The other groups are obviously closely related; the parameter $c$ can be understood as controlling the anisotropy used in the scaling.
 Haar measure on $H$  is given by $db \frac{da}{|a|^{2}}$, the modular function is $\Delta_H(h) = |a|^{c-1}$. The dual orbit is
computed as
\[ \mathcal{O} =
 \mathbb{R}^2 \setminus (\{ 0 \} \times \mathbb{R}))~.\]
For $h = \left( \begin{array}{cc} a & b \\ 0 & a^c
\end{array}
 \right) \in H$ and $\xi_0 = (1,0)^T \in \mathcal{O}$, we obtain $h^T \xi_0 = (a,b)^T$.

 We assume that the weight $w$ on $H$ fulfills 
 \[ w(h) \le
(|a|+|a|^{-1} + |b|)^u ~~. \] Note that the right-hand side is not necessarily submultiplicative (unless $|c| \le 1$), but since we only need an upper estimate of the weight to guarantee temperate embeddedness, this property is immaterial at this point. 
For the shearlet group, corresponding to $c=1/2$, the authors of  
\cite{DaStTe10,DaKuStTe} considered weights of the sort 
\begin{equation} \label{eqn:weight_shearlet} w(h) = |a|^{r_1} (|a|+|a|^{-1}+|a^{-1/2} b|)^{r_2} ~.\end{equation} (When comparing with \cite{DaStTe10,DaKuStTe}, note that the parametrization of the dilation group used in these sources differs from the one used here.)  $w$ defined via (\ref{eqn:weight_shearlet}) fullfills the above growth estimate with  $u = r_1+2r_2$. 

For the matrix norm, the choice $\| h \| =
\max(|a|,|a|^c,|b|)$ is obviously equivalent to $\| h \|_\infty$, and more convenient to use. 
The determinant of $h$ is $|a|^{1+c}$.

For $\xi \in \mathcal{O}$, we have ${\rm dist}(\xi,\mathcal{O}^c) =
|\xi_1|$ and $\sqrt{|\xi|^2-{\rm dist}(\xi,\mathcal{O}^c)^2} = |\xi_2|$. Hence condition (i) in
Definition \ref{defn:mod_embedded} requires finiteness of
\begin{eqnarray*}
\lefteqn{\int_{\mathbb{R}'} \int_{\mathbb{R}} |a|^{(-q/2-1)(1+c)}
(1+\max(|a|,|a|^c,|b|))^{(s+3)q} (|a|+|a|^{-1}+|b|)^{uq}} \\
& \times & \min \left(
\frac{|a|}{1+|b|},\frac{1}{1+|a|+|b|} \right)^{\ell q}~db
\frac{da}{|a|^{2}}~
\end{eqnarray*} for sufficiently large $\ell$. This is easily achieved
by the balancing argument used for the previous cases, since
\[
\max(|a|,|a|^{-1},|b|) \cdot ~\min \left(
\frac{|a|}{1+|b|},\frac{1}{1+|a|+|b|} \right) \le 1~,
\] and $(1+|a|+|b|)^{-2-\epsilon}$ integrable.

For the second condition, we first compute $h^{-1} = \left(
\begin{array}{cc} a^{-1} & - a^{-c-1} b \\ 0 & a^{-c} \end{array}
\right)$. Thus condition (ii) translates to
\begin{eqnarray*}
\lefteqn{\int_{\mathbb{R}'} \int_{\mathbb{R}} |a|^{(-1/2-1/q)(1+c)}
(1+\max(|a|,|b|))^{3} (|a|+|a|^{-1}+|b|)^{u}} \\ & \times & \min \left(
\frac{|a|^{-1}}{1+|a|^{-c-1}|b|},\frac{1}{1+|a|^{-1}+|a|^{-c-1}|b|}
\right)^{\ell}~db\frac{da}{|a|^2}< \infty~.
\end{eqnarray*} To see that this holds for large enough $\ell$, we first note
that \begin{eqnarray} \lefteqn{ |a|^{\pm 1} ~\min \left(
\frac{|a|^{-1}}{1+|a|^{-c-1}|b|},\frac{1}{1+|a|^{-1}+|a|^{-c-1}|b|}
\right) = } \nonumber \\ &= & |a|^{\pm 1} ~\min \left(
\underbrace{\frac{1}{|a|+|a|^{-c}|b|}}_{\le
|a|^{-1}},\underbrace{\frac{|a|}{1+|a|+|a|^{-c}|b|}}_{\le |a|}
\right) \le 1 \label{ineq:a}~.
\end{eqnarray}
In addition,
\[ |b| ~\min \left(
\frac{|a|^{-1}}{1+|a|^{-c-1}|b|},\frac{1}{1+|a|^{-1}+|a|^{-c-1}|b|}
\right) \le \frac{|b|}{|a|+|a|^{-c}|b|} = |a|^c \frac{|b| }{|a|^{1+c}+|b|} \le
|a|^c~,\]  which implies together with inequality (\ref{ineq:a}):
\[
|b| ~\min \left(
\frac{|a|^{-1}}{1+|a|^{-c-1}|b|},\frac{1}{1+|a|^{-1}+|a|^{-c-1}|b|}
\right)^{1+|c|} \le 1~.
\]
Thus the usual course of argument, together with
\begin{eqnarray*}
\lefteqn{\int_{\mathbb{R}'} \int_{\mathbb{R}} \min \left(
\frac{|a|^{-1}}{1+|a|^{-c-1}|b|},\frac{1}{1+|a|^{-1}+|a|^{-c-1}|b|}
\right)^{r} db da} \\ & = & \int_{\mathbb{R}'} \int_{\mathbb{R}}
~\min \left(\frac{1}{|a|+|a|^{-c}|b|},\frac{|a|}{1+|a|+|a|^{-c}|b|}
\right)^r db da \\
& \le & 4 \int_1^\infty \int_0^\infty \left( \frac{1}{a+a^{-c} b}
\right)^r db da + 4 \int_{0}^1 \int_0^\infty \left(
\frac{a}{1+a+a^{-c} b}
\right)^r db da ~\\
& < & \infty~, \end{eqnarray*}  as soon as $r> \max(2,c+2)$, yields finiteness
of the integral for sufficiently large $\ell$, and we have established that the orbit is
$w$-temperately embedded.

\begin{remark}
 The calculations show that the index $\ell$ depends on the anisotropy parameter $c$. Hence, even though the dual orbits coincide for all groups $H_c$, the indices involved in temperate embeddedness (and thus: the number of vanishing moments required for the weighted $L^{p,q}$-estimates) may depend on the group. This serves as a reminder that temperate embeddedness is not strictly a property of the orbit (as a set), but rather a property of the group and the way it acts on the orbit.
\end{remark}

\begin{remark}
The same reasoning as in the previous two cases allows to verify that Corollary \ref{cor:L1_adm_cond_vm} is also applicable to yield sufficient criteria for analyzing vectors. 

This time, powers of the operator $Df = \partial_{x_1} f$ can be used to  induce vanishing moments of a prescribed order.
\end{remark}

\begin{remark}
 As already mentioned, the case $c=1/2$ corresponds to the shearlet group. Thus the results of Section \ref{sect:coorbit} all apply to this group. Many results in this paper have shearlet counterparts in the literature (e.g. in \cite{DaKuStTe,DaStTe10}) that typically are similar in spirit but do not exactly match in terms of the required assumptions. As examples we mention Theorem \ref{thm:ex_int_vec}, Lemma \ref{lem:osc},  Corollary \ref{cor:C_c_O_dense}, Corollary \ref{cor:L1_adm_cond_vm}. I am not aware of a shearlet version of Corollary \ref{cor:coorb_vm}.
\end{remark}

\section*{Acknowledgement}
I would like to thank Karlheinz Gr\"ochenig for illuminating discussions concerning integrable representations, in particular for conjecturing Theorem \ref{thm:ex_int_vec}. Thanks are also due to Felix Voigtlaender for a thorough reading of the paper and many useful comments. 

\bibliography{adm_gen_dil_bib.bib}
\bibliographystyle{plain}
\end{document}